\documentclass{article}
\usepackage{hyperref,graphicx}
\usepackage{authblk}
\usepackage{comment}
\usepackage{microtype}
\usepackage{fullpage}
\usepackage{url}
\usepackage{amsfonts,amscd,amsmath,amssymb,amsthm}
\usepackage[english]{babel}
\usepackage{mathtools}
\usepackage{bbm}

\newcommand{\R}{\mathbb{R}}

\newcommand{\Q}{\mathbb{Q}}
\newcommand{\E}{\mathbb{E}}

\renewcommand{\P}{\mathbb{P}}

\newcommand{\calG}{G}

\newcommand{\calA}{\mathcal{A}}

\newcommand{\normal}{\mathcal{N}}

\newcommand{\pdiffII}[3]{\ifstrequal{#2}{#3}
{\frac{\partial^2 #1}{\partial #2^2}}
{\frac{\partial^2 #1}{\partial #2 \partial #3}}
}
\newcommand{\diffII}[3]{\ifthenelse{\equal{#2}{#3}}
{\frac{d^2 #1}{d #2^2}}
{\frac{d^2 #1}{d #2 d #3}}
}

\newcommand{\toD}{\stackrel{d}{\to}}
\let\Pr\relax
\DeclareMathOperator{\Pr}{Pr}
\DeclareMathOperator{\Pois}{Pois}

\DeclareMathOperator{\Var}{Var}
\DeclareMathOperator{\Cov}{Cov}

\DeclareMathOperator{\tr}{tr}

\let\vec\relax
\DeclareMathOperator{\vec}{vec}
\DeclareMathOperator{\Multinom}{Multinom}
\DeclareMathOperator{\Normal}{\mathcal{N}}
\DeclareMathOperator{\diag}{diag}
\newcommand{\prob}[1]{\mathbb{P}_n\left(#1\right)}
\newcommand{\probcond}[1]{\hat{\mathbb{P}}_n\left(#1\right)}

\newcommand{\expec}[1]{\mathbb{E}\left(#1\right)}
\newcommand{\probER}[1]{\mathbb{Q}_n\left(#1\right)}
\newcommand{\abs}[1]{\left|#1\right|}
\newcommand{\sigtil}{\widetilde{\sigma}}
\newcommand{\tautil}{\widetilde{\tau}}
\newcommand{\Ntil}{\widetilde{N}}
\newcommand{\Xtil}{\widetilde{X}}
\newcommand{\xitil}{\widetilde{\xi}}

\newcommand{\bigoh}[1]{O\left(#1\right)}
\newcommand{\trans}[1]{{#1}^{\intercal}}
\newcommand{\defas}{:=}
\newcommand{\eye}{\mathbb{I}}
\newcommand{\1}{\mathbf{1}}
\newcommand{\indicator}[1]{\mathbbm{1}_{\{#1\}}}
\newcommand{\cin}{{c_{\mathrm{in}}}}
\newcommand{\cout}{{c_{\mathrm{out}}}}
\newcommand{\ein}{{m_{\mathrm{in}}}}
\newcommand{\eout}{{m_{\mathrm{out}}}}
\newcommand{\mein}{{\overline{m}_{\mathrm{in}}}}
\newcommand{\meout}{{\overline{m}_{\mathrm{out}}}}
\newcommand{\mm}{\overline{m}}
\newcommand{\J}{\mathbb{J}}
\newcommand{\e}{\mathrm{e}}

\newcommand{\ER}{Erd\H{o}s-R\'{e}nyi}
\newcommand{\SBM}{\text{SBM}}

\newcommand{\connectivity}{T}
\newcommand{\overlap}{\alpha}

\newcommand{\dcond}{d_{\mathrm{c}}}
\newcommand{\dcondupper}{d_{\mathrm{c}}^{\mathrm{upper}}}
\newcommand{\dcondlower}{d_{\mathrm{c}}^{\mathrm{lower}}}
\newcommand{\avgdeg}{d}
\newcommand{\Bin}{\mathrm{Bin}}

\DeclareMathOperator{\Poisson}{Poisson}
\DeclareMathOperator{\olap}{overlap}
\DeclareMathOperator{\Tr}{tr}

\newtheorem{definition}{Definition}
\newtheorem{theorem}{Theorem}
\newtheorem{lemma}{Lemma}

\newtheorem{proposition}{Proposition}

\begin{document}
\title{Information-theoretic thresholds for community detection in sparse networks}

\author[1]{Jess Banks \thanks{banks.jess.m@gmail.com}}
\author[1]{Cristopher Moore \thanks{moore@santafe.edu}}
\author[2,3]{Joe Neeman \thanks{joeneeman@gmail.com}}
\author[4]{Praneeth Netrapalli \thanks{praneeth@microsoft.com}}

\affil[1]{Santa Fe Institute, Santa Fe NM}
\affil[2]{Institute of Applied Mathematics, University of Bonn, Bonn, Germany}
\affil[3]{Mathematics Department, University of Texas, Austin TX}
\affil[4]{Microsoft Research, Cambridge MA}
\maketitle
\begin{abstract}
We give upper and lower bounds on the information-theoretic threshold for
community detection in the stochastic block model.  Specifically, consider a
symmetric stochastic block model with $q$ groups, average degree $d$, and
connection probabilities $\cin/n$ and $\cout/n$ for within-group and
between-group edges respectively; let $\lambda = (\cin-\cout)/(q\avgdeg)$.  We
show that, when $q$ is large, and $\lambda = O(1/q)$, the critical value of $d$
at which community detection becomes possible---in physical terms, the
condensation threshold---is
\[
\dcond = \Theta\!\left( \frac{\log q}{q \lambda^2} \right) \, , 
\]
with tighter results in certain regimes.  Above this threshold, we show that
any partition of the nodes into $q$ groups which is as `good' as the planted one,
in terms of the number of within- and between-group edges, is correlated with it.
This gives an exponential-time algorithm that performs better than chance; 
specifically, community detection becomes possible below the Kesten-Stigum bound 
for $q \ge 5$ in the disassortative case $\lambda < 0$, and for $q \ge 11$ in the 
assortative case $\lambda > 0$  
(similar upper bounds were obtained independently by Abbe and Sandon).  
Conversely, below this threshold, we show that no algorithm can label the vertices better 
than chance, or even distinguish the block model from an \ER\ random graph with 
high probability.

Our lower bound on $\dcond$ uses Robinson and Wormald's small subgraph conditioning method,
and we also give (less explicit) results for non-symmetric stochastic block
models.  In the symmetric case, we obtain explicit results by using bounds
on certain functions of doubly stochastic matrices due to Achlioptas and Naor;
indeed, our lower bound on $\dcond$ is their second moment lower bound on the
$q$-colorability threshold for random graphs with a certain effective
degree.
\end{abstract}

 \section{Introduction}
The Stochastic Block Model (SBM) is a random graph ensemble with planted
community structure, where the probability of a connection between each pair of
vertices is a function only of the groups or communities to which they belong.
It was originally invented in sociology (\cite{HLL83}); it was reinvented in
physics and mathematics under the name ``inhomogeneous random
graph'' (\cite{Soderberg02,BJR07}), and in computer science as the planted
partition problem (e.g.~\cite{mcsherry}).  

Given the current interest in network science, the block model and its variants
have become popular parametric models for the detection of community structure.
An interesting set of questions arise when we ask to what extent the
communities, i.e., the labels describing the vertices' group memberships, can
be recovered from the graph it generates.  In the case where the average degree
grows as $\log n$, if the structure is sufficiently strong then the underlying
communities can be recovered (\cite{BC09}), and the threshold at which this
becomes possible has recently been
determined (\cite{abbe-bandeira-hall,abbe-sandon,agarwal-etal}).  Above this
threshold, efficient algorithms exist that recover the communities exactly,
labeling every vertex correctly with high probability; below this threshold,
exact recovery is information-theoretically impossible. 

In the sparse case where the average degree is $O(1)$, finding the communities
is more difficult, since we effectively have only a constant amount of
information about each vertex.  In this regime, our goal is to label the
vertices better than chance, i.e., to find a partition with nonzero correlation
or mutual information with the ground truth.  This is sometimes called the
\emph{detection} problem to distinguish it from exact recovery.  A set of phase
transitions for this problem was conjectured in the statistical physics
literature based on tools from spin glass theory (\cite{DKMZ11a,DKMZ11b}).  Some of these conjectures
have been made rigorous, while others remain as tantalizing open problems.

Besides the detection problem, it is natural to ask whether a graph generated by the stochastic block model
can be distinguished from an Erd\H{o}s-R\'enyi random graph with the same average degree.
This is called the \emph{distinguishability} problem, and it is believed to have
the same threshold as the detection problem. Although distinguishing a graph from the stochastic
block model from an Erd\H{o}s-R\'enyi graph seems intuitively easier than actually detecting the
communities, we do not know any rigorous proof of this statement.

\subsection{The Kesten-Stigum bound, information-theoretic detection, and condensation}

Although we will also deal with non-symmetric stochastic block models, in this
discussion we focus on the symmetric case where the $q$ groups are of equal expected size, 
and the probability of edges between
vertices within and between groups are $\cin/n$ and $\cout/n$ respectively for constants $\cin, \cout$.  
The expected average degree of the resulting graph is then
\begin{equation}
\label{eq:avgdeg}
\avgdeg = \frac{\cin + (q-1)\cout}{q} \, . 
\end{equation}
It is convenient to parametrize the strength of the community structure as 
\begin{equation}
\label{eq:lambda}
\lambda = \frac{\cin-\cout}{q\avgdeg} \, .
\end{equation}
As we will see below, this is the second eigenvalue of a transition matrix describing how labels are ``transmitted'' between neighboring vertices.  It lies in the range
\[
-\frac{1}{q-1} \le \lambda \le 1 \, ,
\]
where $\lambda = -1/(q-1)$ corresponds to $\cin = 0$ (also known as the planted
graph coloring problem) and $\lambda = 1$ corresponds to $\cout = 0$
where vertices only connect to others in the same group.  We say that block
models with $\lambda > 0$ and $\lambda < 0$ are \emph{assortative} and
\emph{disassortative} respectively.

The conjecture of~\cite{DKMZ11a,DKMZ11b} is that efficient algorithms exist if and only if we are above the threshold
\begin{equation}
\label{eq:kesten-stigum}
d = \frac{1}{\lambda^2} \, .
\end{equation}
This is known in information theory as the Kesten-Stigum
threshold (\cite{Kesten1966,Kesten1966a}), and in physics as the Almeida-Thouless
line (\cite{AlmeidaThouless78}).

Above the Kesten-Stigum threshold, \cite{DKMZ11a,DKMZ11b} claimed that
community detection is computationally easy, and moreover that belief
propagation---also known in statistical physics as the cavity method---is
asymptotically optimal in that it maximizes the fraction of vertices labeled
correctly (up to a permutation of the groups).  For $q=2$, this was proved
in~\cite{MNS-colt}; very recently \cite{abbe-sandon-more-groups} showed
that a type of belief propagation performs better than chance for all $q$. In
addition,~\cite{bordenave-lelarge-massoulie} showed that a spectral clustering
algorithm based on the non-backtracking operator succeeds all the way down to
the Kesten-Stigum threshold (proving a conjecture of~\cite{Krzakala13}, who
introduced the algorithm).

What happens below the Kesten-Stigum threshold is more complicated.
\cite{DKMZ11a,DKMZ11b} conjectured that for sufficiently small $q$, community
detection is information-theoretically impossible when $d < 1/\lambda^2$.
\cite{MNS12} proved this in the case $q=2$: first, they showed that the
ensemble of graphs produced by the stochastic block model becomes
\emph{contiguous} with that produced by \ER\ graphs of the same average degree,
making it impossible even to tell whether or not communities exist with high
probability.  Secondly, by relating community detection to the Kesten-Stigum
reconstruction problem on trees (\cite{EKPS}), they showed that for
most pairs of vertices the probability, given the graph, that they are in the
same group asymptotically approaches $1/2$.  Thus it is impossible, even if we
could magically compute the true posterior probability distribution, to label
the vertices better than chance. 

On the other hand, \cite{DKMZ11a,DKMZ11b} conjectured that for sufficiently
large $q$, namely $q \ge 5$ in the assortative case $\cin > \cout$ and $q \ge
4$ in the disassortative case $\cin < \cout$, there is a ``hard but
detectable'' regime where community detection is information-theoretically
possible, but computationally hard.  One indication of this is the extreme case
where $\cin = 0$: this is equivalent to the planted graph coloring problem
where we choose a uniformly random coloring of the vertices, and then choose $dn/2$ edges
uniformly from all pairs of vertices with different colors.  In this case, we
have $\lambda = -1/(q-1)$ and~\eqref{eq:kesten-stigum} becomes $d > (q-1)^2$.
However, while graphs generated by this case of the block model are
$q$-colorable by definition, the $q$-colorability threshold for \ER\ graphs
grows as $2q \ln q$ (\cite{achlioptas-naor}), and falls below the Kesten-Stigum
threshold for $q \ge 5$.  In between these two thresholds, we can at least
distinguish the two graph ensembles by asking whether a $q$-coloring exists;
however, finding one might take exponential time. 

More generally, planted ensembles where some combinatorial structure is built
into the graph, and un-planted ensembles such as \ER\ graphs where these
structures occur by chance, are believed to become distinguishable at a phase transition called
\emph{condensation} (\cite{Krzakala2007a}).  Below this point, the two ensembles
are contiguous; above it, the posterior distribution of the partition or coloring conditioned 
on the graph---in physical terms, the Gibbs distribution---is
dominated by a cluster of states surrounding the planted state.  For instance,
in random constraint satisfaction problems, the uniform distribution on
solutions becomes dominated by those near the planted one; in our setting, the
posterior distribution of partitions becomes dominated by those close to the
ground truth (although, in the sparse case, with a Hamming distance that is
still linear in $n$).  Thus the condensation threshold is believed to be the threshold
for information-theoretic community detection.  Below it, even optimal Bayesian
inference will do no better than chance, while above it, typical partitions
chosen from the posterior will be fairly accurate (though finding these typical
partitions might take exponential time).  

We note that some previous results show that community detection is possible
below the Kesten-Stigum threshold when the sizes of the groups are
unequal (\cite{zhang-moore-newman}).  In addition, even a
vanishing amount of initial information can make community detection possible
if the number of groups grows with the size of the
network (\cite{kanade-mossel-schramm}).

\subsection{Our contribution}
\label{sec:intro-contribution}

We give rigorous upper and lower bounds on the condensation threshold.  Our
bounds are most explicit in the case of symmetric stochastic block models, in
which case we give upper and lower bounds for the condensation threshold
as a function of $q$ and
$\lambda$.  First, we use a first-moment argument to show that if 
\begin{equation}
\label{eq:d-upper}
d > \dcondupper 
= \frac{2 q \log q}
{(1+(q-1) \lambda) \log (1+(q-1) \lambda) 
+ (q-1)(1-\lambda) \log (1-\lambda)} \, ,
\end{equation}
then, with high probability, the only partitions that are as good as the
planted one---that is, which have the expected number of edges within and
between groups---have a nonzero correlation with the planted one.  As a result,
there is a simple exponential-time algorithm for labeling the vertices better
than chance: simply test all partitions, and output the first good one.  

We note that $\dcondupper < 1/\lambda^2$ for $q \ge 5$ when $\lambda$ is
sufficiently negative, including the case $\lambda = -1/(q-1)$ corresponding to
graph coloring discussed above.  Moreover, for $q \ge 11$, there also exist
positive values of $\lambda$ for $\dcondupper < 1/\lambda^2$.  Thus for
sufficiently large $q$, detectability is information-theoretically possible
below the Kesten-Stigum threshold, in both the assortative and disassortative
case.  Similar (and somewhat tighter) results were obtained independently
by~\cite{abbe-sandon-isit}.

We then show that community detection is information-theoretically impossible if
\begin{equation}
\label{eq:d-lower}
d < \dcondlower 
= \frac{2\log(q-1)}{q-1} \frac{1}{\lambda^2}
\, . 
\end{equation}
Using the small
subgraph conditioning method, we show that the block model and the \ER\
graph are contiguous whenever the second moment of the ratio between their
probabilities---roughly speaking, the number of good partitions in an \ER\
graph---is appropriately bounded.  We 	also show that this second moment bound
implies non-detectability, in that the posterior distribution on any finite
collection of vertices is asymptotically uniform.  This reduces the proof of
contiguity and non-detectability to a second moment argument; in the case of a
symmetric stochastic block model, this consists of maximizing a certain
function of doubly stochastic matrices.  

Happily, this latter problem was largely solved by \cite{achlioptas-naor}, who
used the second moment method to give nearly tight lower bounds on the
$q$-colorability threshold.  Our bound~\eqref{eq:d-lower} corresponds
to their lower bound on $q$-colorability for $G(n,\avgdeg'/n)$ where $\avgdeg'
= \avgdeg \lambda^2 (q-1)^2$.  Intuitively, $\avgdeg'$ is the degree of a
random graph in which the correlations between vertices in the $q$-colorability
problem are as strong as those in the stochastic block model with average degree $\avgdeg$
and eigenvalue $\lambda$.

Our bounds are tight in some regimes, and rather loose in others.  Let $\mu$ denote $(\cin-\cout)/\avgdeg$.  If $\mu$ is constant and $q$ is large, we have
\[
\lim_{q \to \infty} \frac{\dcondupper}{\dcondlower}
= \frac{\mu^2}{(1+\mu) \log (1+\mu) - \mu} \, .
\]
In the limit $\mu = -1$, corresponding to graph coloring, this ratio is $1$, inheriting the tightness of previous upper and lower bounds on $q$-colorability.  For other values of $\mu$, our bounds match up to a multiplicative constant.  In particular, when $q$ is constant and $|\lambda|$ is small, they are about a factor of $2$ apart:
\[
\frac{2\log(q-1)}{q-1} \le \dcond \lambda^2 \le \frac{4 \log q}{q-1} (1+O(q \lambda)) \, . 
\]
When $\lambda \ge 0$ is constant and $q$ is large, we have 
\[
\dcondupper = \frac{2}{\lambda} (1 + O(1/\log q)) \, .
\]
Thus, in the limit of large $q$, detectability is possible below the Kesten-Stigum threshold whenever $\lambda < 1/2$.

\section{Definitions and results}

A stochastic block model with $q \ge 2$ communities is parametrized by two quantities: the distribution
$\pi \in \Delta_{q}$ of vertex classes and the symmetric matrix $M \in \R^{q \times q}$ of edge
probabilities. Given these two parameters, a random graph from the block model
$\calG(n, M/n, \pi)$ is generated as follows: for each vertex $v$, sample
a label $\sigma_v$ in $[q] = \{1, \dots, q\}$ independently with distribution $\pi$. Then,
for each pair $(u, v)$, include the edge $(u, v)$ in the graph independently with probability
$n^{-1} M_{\sigma_u,\sigma_v}$. Since we will worq with a fixed $M$ and $\pi$ throughout,
we denote $\calG(n, M/n, \pi)$ by $\P_n$. Note that according to the preceding description,
we have the following explicit form for the density of $\P_n$:
\[
\P_n(G, \sigma) = \prod_{v \in V(G)} \pi_{\sigma_v}
\prod_{(u, v) \in E(G)} \frac{M_{\sigma_u,\sigma_v}}{n}
\prod_{(u, v) \not \in E(G)} \left(1 - \frac{M_{\sigma_u,\sigma_v}}{n}\right).
\]
We will assume throughout that every vertex in $G \sim \P_n$ has the same expected degree.
(In terms of $M$ and $\pi$, this means that $\sum_j M_{ij} \pi_j$ does not depend on $i$.)
Without this assumption, reconstruction and distinguishability -- at least in the way
that we will define them -- are trivial, since we gain non-trivial information on
the class of a vertex just by considering its degree.

With the preceding assumption in mind, let $d = \sum_j M_{ij} \pi_j$ be the expected
degree of an arbitrary vertex. In order to discuss distinguishability, we will compare
$\P_n$ with the Erd\H{o}s-R\'enyi distribution $\Q_n := \calG(n, d/n)$.

Throughout this work, we will make use of the matrix $T$ defined by
\[
T_{ij} = \frac 1d \pi_i M_{ij}, \label{eq:T-def}
\]
or in other words, $T = \frac 1d \diag(\pi) M$. Note that $T$ is a stochastic matrix, in the
sense that it has non-negative elements and all its rows sum to 1. The Perron-Frobenius eigenvectors
of $T$ are $\pi$ on the right, and $\1$ on the left (where $\1$ denotes the vector of ones),
and the corresponding eigenvalue is 1. We let $\lambda_1, \dots, \lambda_q$ be the eigenvalues of
$T$, arranged in order of decreasing absolute value (so that $\lambda_1 = 1$ and $|\lambda_2| \le 1$).
The second of these turns out to be the most important for us; therefore, set $\lambda = \lambda_2$.

There is an important probabilistic interpretation of the matrix $T$ relating to the local
structure of $G \sim \P_n$; although we will not rely on this interpretation
in the current work, it played an important role in~\cite{MoNeSl:13}. Indeed,
one can show that for any fixed radius $R$,
the $R$-neighborhood of a vertex in $G \sim \P_n$ has almost the same distribution as a Galton-Watson
tree with radius $R$ and offspring distribution $\Poisson(d)$. Then, the class labels on the
neighborhood can be generated by first choosing the label of the root according to $\pi$ and then,
conditioned on the root's label being $i$, choosing its children's labels independently to be
$j$ with probability $T_{ij}$. This procedure continues down the tree: any vertex with parent $u$
has probability $T_{\sigma_u j}$ to receive the label $j$. Thus, $T$ is the transition matrix
of a certain Markov process that describes a procedure for approximately generating the class
labels on a local neighborhood in $G$.

In part of this work, we will deal with the symmetric case, in which $\pi_i = \frac 1q$ for all $i$ and
\begin{equation}
\label{eq:c-cin-cout}
M_{i,j} = \begin{cases} \cin & \mbox{if $i=j$} \\ \cout & \mbox{if $i \ne j$} \, . 
\end{cases}
\end{equation}
In this case, the expected average degree is
\[
\avgdeg = \frac{\cin + (q-1)\cout}{q} \, ,
\]
the Markov transition matrix (which is symmetric, and hence doubly stochastic) is
\begin{equation}
	\connectivity
	= \frac{1}{q\avgdeg} \begin{pmatrix}
			\cin & {} & \cout \\
			{} & \ddots & {} \\
			\cout & {} & \cin
	\end{pmatrix}
	= \lambda \eye + (1-\lambda) \frac{\J}{q},\,
	\label{eq:connectivity}
\end{equation}
where $\eye$ is the identity matrix, $\J$ is the matrix of all $1$s, and where 
\[
\lambda = \frac{\cin-\cout}{q\avgdeg} 
\]
is $\connectivity$'s second eigenvalue.
% Its second eigenvalue is
% \[
%   \lambda = \frac{\cin - \cout}{kd}.
% \]
We can think of $\lambda$ as the probability that information is transmitted
from $u$ to $v$: with probability $\lambda$ we copy $u$'s group label to $v$,
and with probability $1-\lambda$ we choose $v$'s group uniformly from $[q]$.  The
parameter $\lambda$ interpolates between the case $\lambda=1$ where all edges
are within-group, to an \ER\ graph where $\lambda = 0$ and edges are placed
uniformly at random, to $\lambda < 0$ where edges are more likely between
groups than within them.  This gives a useful reparametrization of the model in
terms of $c$ and $\lambda$, where
\begin{align}
	\cin &= \avgdeg (1 + (q-1) \lambda) \nonumber \\
	\cout &= \avgdeg (1 - \lambda) \, .
	\label{eq:reparam}
\end{align}

  For labellings $\sigma$ and $\tau$ in $[q]^n$, define their \emph{overlap} by
\[
\olap(\sigma, \tau) = \frac{1}{n}\max_\rho \sum_{i=1}^q \left(|\sigma^{-1}(i) \cap \tau^{-1}(\rho(i))| - \frac 1n |\sigma^{-1}(i)| |\tau^{-1} (\rho(i))| \right),
\]
where the supremum runs over all permutations $\rho$ of $[q]$. In words, $\sigma$ and $\tau$
have a positive overlap if there is some relabelling of $[q]$ so that they are positively correlated.

\begin{definition}
We say that the block model $\P_n = \calG(n, M/n, \pi)$ is \emph{detectable}
if there is some $\delta > 0$ and an algorithm $\calA$ mapping graphs to labellings
such that if $(G, \sigma) \sim \P_n$ then
\[
 \lim_{n \to \infty} \Pr(\olap(\calA(G), \sigma) > \delta) > 0.
\]
\end{definition}

\begin{definition}
We say that $\P_n$ and $\Q_n$ are \emph{asymptotically orthogonal} if there is a sequence $A_n$ of events such
that $\P_n(A_n) \to 0$ and $\Q_n(A_n) \to 1$.

We say that $\P_n$ and $\Q_n$ are \emph{contiguous} if for every sequence $A_n$ of events,
$\P_n(A_n) \to 0$ if and only if $\Q_n(A_n) \to 0$.
\end{definition}

Our main result is the following:
\begin{theorem}\label{thm:main}
  Consider the symmetric stochastic block model $\P_n$ with $q$ communities, average degree $\avgdeg$,
  and second-eigenvalue $\lambda$. Define
\begin{align}
\dcondupper 
&= \frac{2 q \log q}
{(1+(q-1) \lambda) \log (1+(q-1) \lambda) 
+ (q-1)(1-\lambda) \log (1-\lambda)} \\
\dcondlower 
&= \frac{2\log(q-1)}{q-1}\frac{1}{\lambda^2} \, .
\end{align}
If $\avgdeg > \dcondupper$ then $\P_n$ and $\Q_n$ are asymptotically orthogonal, and $\P_n$ is detectable.
If $\avgdeg < \dcondlower$ then $\P_n$ and $\Q_n$ are contiguous, and $\P_n$ is not detectable.
\end{theorem}

The lower bound in Theorem~\ref{thm:main} comes from a more general (but less explicit) bound that holds also
for block models that are not symmetric. In order to state the more general result, we must first introduce
some notation.

\begin{definition}\label{def:D}
 Let $\Delta_m$ denote the probability simplex in $\R^m$:
\begin{align*}
  \Delta_m \defas \{p \in \R^m:  p_i \geq 0, \sum_{i=1}^m p_i = 1\}.
\end{align*}
   Define $D: \Delta_m \times \Delta_m \to \R$ by
  \[
   D(p, \tilde p) = \sum_{i=1}^m p_i \log (p_i/\tilde p_i).
  \]
\end{definition}
 
Note that if we interpret $p, \tilde p \in \Delta_m$ as probability distributions on a $m$-point
set, then $D(p, \tilde p)$ is exactly the Kullback-Leibler divergence of $p$ with respect to $\tilde p$.

\begin{definition}
For $\pi \in \Delta_q$, define
\[
 \Delta_{q^2}(\pi) \defas \{
 (p_{ij})_{i,j=1}^q \in \Delta_{q^2}: \sum_{i=1}^q p_{ij} = \pi_j \text{ and } \sum_{j=1}^q p_{ij} = \pi_i
 \text{ for all } i,j
 \}.
\]
In other words, elements of $\Delta_{q^2}(\pi)$ are probability distributions on $[q]^2$ that have $\pi$ as
their marginal distributions.
\end{definition}

\begin{definition}\label{def:Q}
For $\pi \in \Delta_q$ and a $q \times q$ matrix $A$, let $p = \pi \otimes \pi$, where $\otimes$ denotes Kronecker product and define
\[
 Q(\pi, A) = \sup_{\alpha \in \Delta_{q^2}(\pi)} \frac{\trans{(\alpha - p)} (A \otimes A) (\alpha - p)}{D(\alpha, p)}.
\]
\end{definition}

Although we do not know any simple algebraic expression for $Q$, one can easily
compute numerical approximations.
For non-symmetric stochastic block models, our main result is that $Q$ gives
a lower bound on the detectability threshold:

\begin{theorem}\label{thm:non-distinguish}
Let $\P_n = \calG(n, M/n, \pi)$ and $\Q_n = \calG(n, d/n)$, where $d = \sum_{j} M_{ij} \pi_j$.
If
\[Q(\pi, (M - d \J)/\sqrt{2d}) < 1\]
then $\P_n$ and $\Q_n$ are contiguous and $\P_n$ is non-detectable.
\end{theorem}

For comparison with the Kesten-Stigum bound, note that $Q(\pi, (M - d\J)/\sqrt{2d}) < 1$
implies that $\lambda^2 d < 1$. This comes from comparing the second derivatives at $p$ in
the numerator and denominator of $Q$: if $Q < 1$ then the Hessian of the numerator must
be smaller (in the semidefinite order) than that of the denominator, and this turns out to be equivalent
to $\lambda^2 d < 1$.

We remark that while $Q(\pi, (M - d\J)/\sqrt{2d}) < 1$ is only a sufficient condition for the contiguity
of $\P_n$ and $\Q_n$, it is actually a sharp condition for a certain second moment to exist:
\begin{proposition}\label{prop:second-moment}
Fix a sequence $a_n$ with $a_n = o(n)$ and $a_n = \omega(\sqrt n)$. Let $\Omega_n$
be the event that for all $i \in [q]$, $|\sigma^{-1}(i)| = n\pi_i \pm a_n$.
With the notation of Theorem~\ref{thm:non-distinguish},
take $\hat \P_n$ to be $\P_n$ conditioned on $\Omega_n$.
If $Q(\pi, (M - d\J)/\sqrt{2d}) < 1$ then
\begin{equation}\label{eq:second-moment}
\lim_{n\to\infty} \E_{\Q_n} \left(\frac{\hat \P_n}{\Q_n}\right)^2 = (1 + o(1)) \prod_{i,j=2}^q \psi(d\lambda_i \lambda_j)
< \infty,
\end{equation}
where $\lambda_1, \cdots, \lambda_q$ are the eigenvalues of $T$ (cf.~\eqref{eq:T-def}) such that $1=\lambda_1 \geq \abs{\lambda_2} \geq \cdots \geq \abs{\lambda_q}$, and $\psi(x) = (1-x)^{-1/2} e^{-x/2-x^2/4}$.
On the other hand, if $Q(\pi, (M-d\J)/\sqrt{2d}) > 1$ then
 \[\lim_{n\to\infty} \E_{\Q_n} \left(\frac{\hat \P_n}{\Q_n}\right)^2 = \infty.\]
\end{proposition}

\subsection{Outline of the paper}
We prove the upper bound of Theorem~\ref{thm:main} in Section~\ref{sec:first}. In Section~\ref{sec:second},
we prove the lower bound of Theorem~\ref{thm:main} assuming Theorem~\ref{thm:non-distinguish}. In Section~\ref{sec:secondmoment},
we prove Proposition~\ref{prop:second-moment}. Finally, in Section~\ref{sec:proof-mainthm}, we prove
Theorem~\ref{thm:non-distinguish}. Some auxiliary results are proved in Appendix~\ref{sec:UI-multinomials}.

\subsection{Outline of the proofs}
The part of Theorem~\ref{thm:main} regarding $\dcondupper$ follows from union bounds.
First, note that under $\P_n$, groups in the planted partition have average in-degree of about
$\cin/k$ and average out-degree of about $(k-1)\cout/k$. We call such partitions ``good.''
In order
to show orthogonality, we show that with high probabability, graphs from $\Q_n$ have no good partitions.
(That is, the events $A_n$ witnessing orthogonality are $A_n = \{G \text{ has no good partitions}\}$.)
We show this by computing the probability that a particular partition is good and comparing it to the
number of all partitions.
In order to show detectability, we show that with high probability under $\P_n$, every good partition
is correlated with the planted partition: we bound the probability that a given partition is good, and
sum the probabilities over all partitions that are uncorrelatd with the planted one.

The part of Theorem~\ref{thm:main} regarding $\dcondlower$ follows from Theorem~\ref{thm:non-distinguish}.
We recognize that the optimization problem in the definition
of $Q$ may be written as an optimization over the set of doubly-stochastic matrices.
Using tools due to \cite{achlioptas-naor} (Theorem~\ref{thm:achlioptas-naor} and Lemma~\ref{lem:achlioptas-naor}), we prove
that $d < \dcondlower$ implies that $Q < 1$, and we conclude by applying Theorem~\ref{thm:non-distinguish}.

Proposition~\ref{prop:second-moment} is the main technical step in the proof of contiguity in
Theorem~\ref{thm:non-distinguish}. With Proposition~\ref{prop:second-moment} in hand,
we apply the small subgraph conditioning method 
(see Theorem~\ref{thm:conditional-second-mom})  
which is a type of conditional second moment method.  
In order to apply it, we only need to know the limiting distribution of small subgraphs under $\hat \P_n$ and $\Q_n$ 
(which are already known) and~\eqref{eq:second-moment} from Proposition~\ref{prop:second-moment}.

The proof of Proposition~\ref{prop:second-moment} itself is tedious but elementary: we expand the
square and write the result as the exponential of a quadratic form in multinomial random variables.
Shifted and renormalized, the multinomial variables have a Gaussian limit; the expectation of
an exponentiated quadratic form of Gaussian variables can be computed exactly, and gives~\eqref{eq:second-moment}.
In order to apply the central limit theorem in the above argument, one needs to check
that the exponentiated quadratic form in multinomial variables is uniformly integrable. This
naturally leads to the condition on $Q$: we need to compare an exponentiated quadratic form
with the multinomial probability mass function, which is essentially an exponentiated entropy.
In the end, we need the entropy to dominate the quadratic form (which is exactly what happens
with $Q < 1$).

Finally, to prove non-detectability in Theorem~\ref{thm:non-distinguish} we compare the distribution
$\P_n$ to the distribution (call it $\tilde \P_n$) obtained by conditioning on the labels of a
constant number of vertices. If we can show that the resulting distributions are close in total variation,
it implies that the labels of those vertices cannot be statistically inferred.
Applying the Cauchy-Schwarz inequality to the total variation distance, it is enough to show that
\[
  \E_{\Q_n} \mathbbm{1}_{\Omega_n} \left(\frac{\P_n}{\Q_n} - \frac{\tilde \P_n}{\Q_n}\right)^2
\]
is small. This naturally leads to a computation very similar to the proof of Proposition~\ref{prop:second-moment}.
The only difference is that we are now conditioning on the labels of a constant number of vertices,
but that has very little effect.

\subsection{Conclusions and future work}
\label{sec:conclusion}

We (and, independently, \cite{abbe-sandon-isit}) have shown that community detection is information-theoretically possible below the Kesten-Stigum threshold.  However, we have not given any evidence that it is computationally hard.  Of course, we cannot hope to prove this without knowing that $\mathrm{P} \ne \mathrm{NP}$, but we could hope to prove that certain classes of algorithms take exponential time.  In particular, we could show that Monte Carlo algorithms or belief propagation take exponential time to find a good partition, assuming their initial states or messages are uniformly random.  

Physically, we believe this occurs because there is a free energy barrier between a ``paramagnetic'' phase of partitions which are essentially random, and a ``ferromagnetic'' or ``retrieval'' phase which is correlated with the planted partition (\cite{DKMZ11a,DKMZ11b,zhang-moore}).  Proving this seems within reach: rigorous results have been obtained in random constraint satisfaction problems (\cite{achlioptas-coja-oghlan,coja-oghlan-efthymiou}) showing that solutions become clustered with $O(n)$ Hamming distance and $O(n)$ energy barriers between them.  In particular, Markov chain Monte Carlo algorithms for sampling the posterior distribution, such as Metropolis-Hastings or Glauber dynamics that update the label of one vertex at time according to its marginal distribution conditioned on the current labels of its neighbors, take exponential time to travel from one cluster to another.  The goal in this case would be to show in a planted model that Monte Carlo takes exponential time to find the cluster corresponding to the planted solution.  

Finally, both our upper and lower bounds can be improved.  Our upper bound requires that w.h.p.\ all good partitions are correlated with the planted one.  We could obtain better bounds by requiring that this is true w.h.p.\ of \emph{most} good partitions, which would require a lower bound on the typical number of good partitions with large overlap.  In the limit $\lambda \to 1$ of strong assortative structure, for instance, one can use the fact that vertices of degree $1$ can be set to match their neighbors, or set freely to give the same typical overlap as the planted partition.  Using these and other ideas,~\cite{abbe-sandon-isit} showed that $\dcond \to 1$ as $\lambda \to 1$, while our bounds only give $\dcond \le 2$.  (For regimes where $\dcond$ is large, their bounds and ours are asymptotically equivalent.)  Further improvements seem possible.  

%When $\dcond$ is small enough, we can improve our bounds by noting that, for instance, vertices of degree $1$ can be set so that they are always the same as their neighbor.  Similarly, one can make improvements by focusing on the giant component or the 2-core.  For instance, for $\lambda=1$ (i.e., $\cout = 0$) our bounds give $\dcond \ge 2$, but the correct result is $\dcond = 1$, since for any $\avgdeg > 1$ the graph has w.h.p.\ two giant components, one in each group.

The second moment lower bound could be improved as it was for the $k$-colorability threshold in~\cite{coja-oghlan-vilenchik}.  Indeed, the condensation threshold $\dcond$ for $k$-coloring was determined exactly in~\cite{bapst-condensation-coloring} for sufficiently large $k$.  It is entirely possible that their techniques could work here.  Note that constraint satisfaction problems correspond to zero-temperature models in physics, while the block model with $\cin, \cout \ne 0$ corresponds to a spin system at positive temperature; but some rigorous results have recently been obtained here as well by~\cite{bapst-positive-temperature}.

%\subsection{An example showing looseness}\label{sec:intro-examples}
%Figure~\ref{fig:unbalanced2clusters} shows that as the clusters become very unbalanced (i.e., $p \rightarrow 0$),
%Theorem~\ref{thm:non-distinguish} only guarantees contiguity when $d \lambda^2$ is very small.
%One might ask whether the true threshold has this behavior.
%Proposition~\ref{prop:reconstruction-belowKS} shows that it does for some models.
%We now present another family of examples for which the behavior of the threshold is quite different:
%\begin{proposition}\label{prop:ex1}
%Consider the block model given by
%\begin{align*}
 %M = d \begin{pmatrix}
	  %\frac{1}{p} & 0 \\
	  %0 & \frac{1}{1-p}
       %\end{pmatrix}.
%\end{align*}
%If $d < 1$ then the above model is not reconstructable.
%\end{proposition}
%Note that the above block model has $\lambda_2 = 1$. Hence Proposition~\ref{prop:ex1} shows that there exist arbitrarily
%unbalanced models, where we can not reconstruct for $d \lambda_2^2 < 1$.

%The main idea behind Proposition~\ref{prop:ex1} is that for $d<1$, the largest component is of size $\bigoh{\log n}$.
%Even if we can reconstruct the labels of nodes in the same cluster very well, we can not predict the labels of nodes in
%different clusters better than random guessing.
%We use a lemma from Mossel et al. \cite{MoNeSl:13} that captures this intuition.

\section{Upper bound for symmetric SBMs: Proof of upper bound in Theorem~\ref{thm:main}}
\label{sec:first}

In this section, we prove the part of Theorem~\ref{thm:main} relating to $\dcondupper$.
Recall that Theorem~\ref{thm:main} assumes a symmetric block model; i.e., $\pi_i = 1/q$ for every $i$, and the connectivity matrix $M$ is determined by only two parameters, $\cin$ and $\cout$.

Our upper bound on the detectability threshold hinges on the following observation.  We say a partition is \emph{balanced} if it has $n/q$ vertices in each group.  With high probability, a graph generated by the SBM has at least one balanced partition, close to the the planted one, where the number of within-group and between-group edges $\ein$ and $\eout$ are close to their expectations.  That is, 
\begin{equation}
\label{eq:in-close}
|\ein - \mein| < n^{2/3} 
\quad \text{and} \quad 
|\eout - \meout| < n^{2/3}
\end{equation}
where
\begin{align} 
\mein &= \frac{\cin}{2q} \,n
= \frac{\avgdeg (1+(q-1)\lambda)}{2q} \,n \nonumber \\
\meout &= \frac{(q-1) \cout}{2q} \,n
= \frac{\avgdeg (q-1)(1-\lambda)}{2q} \,n \, .
\label{eq:mein-def}
\end{align}
%where $\approx$ denotes errors of size $o(n^{-1/3})$.  
This follows from standard concentration inequalities on the binomial
distribution: the number of vertices in each group in $\sigma$ is w.h.p.\ $n/q
+ o(n^{2/3} / \log n)$, in which case~\eqref{eq:in-close} holds w.h.p.  Since
the maximum degree is w.h.p.\ less than $\log n$, we can modify $\sigma$ to
make it balanced while changing $\ein$ and $\eout$  by $o(n^{2/3})$.

Call such a partition \emph{good}.  We will show that if $\avgdeg > \dcondupper$ all good partitions are correlated with the planted one.  As a result, there is an exponential algorithm that performs better than chance: simply use exhaustive search to find a good partition, and output it.

\subsection{Distinguishability from $G(n,\avgdeg/n)$}

As a warm-up, we show that if $\avgdeg > \dcondupper$ the probability that an \ER\ graph has a good partition is exponentially small, so the two distributions $\P$ and $\Q$ are asymptotically orthogonal.

Let $G$ be a graph generated by $G(n,\avgdeg/n)$.  We condition on the high-probability event that it has $m$ edges with $|m - \mm| < n^{2/3}$ with
\[
\mm = \mein + \meout = \avgdeg n/2 \, ,
\]
in which case $G$ is chosen from $G(n,m)$.  Since $G$ is sparse, we can think of its $m$ edges as chosen uniformly with replacement from the $n^2$ possible ordered pairs.  With probability $\Theta(1)$ the resulting graph is simple, with no self-loops or multiple edges, and hence uniform in $G(n,m)$.  Thus any event that holds with high probability in the resulting model holds with high probability in $G(n,m)$ as well.  Call this model $G'(n,m)$.  

For a given balanced partition $\sigma$, the probability in $G'(n,m)$ that a given edge has its endpoints in the same group is $1/q$.  Thus, up to subexponential terms resulting from summing over the $n^{2/3}$ possible values of the error terms, the probability that a given $\sigma$ is good is 
\[
\Pr[\Bin(\mm,1/q) = \mein] = {\mm \choose \mein} (1/q)^\mein (1-1/q)^{\meout} \, . 
\]
The rate of this large-deviation event is given by the Kullback-Leibler divergence between binomial distributions with success probability $1/q$ and $\mein/\mm$, 
\begin{align*}
\lim_{\mm \to \infty} \frac{1}{\mm} \log \Pr[\Bin(\mm,1/q) = \mein]
&= -\frac{\mein}{\mm} \log \frac{\mein/\mm}{1/q} - \frac{\meout}{\mm} \log \frac{\meout/\mm}{1-1/q} \\
&= -\frac{\cin}{q \avgdeg} \log \frac{\cin}{\avgdeg} - \left( 1 - \frac{\cin}{q \avgdeg} \right) \log \frac{q \avgdeg-\cin}{\avgdeg (q-1)} \, , 
\end{align*}
where we used $\mein/\mm = \cin/(q \avgdeg)$ and $\meout/\mm = 1-\cin/(q \avgdeg)$.  Writing this in terms of $\avgdeg$ and $\lambda$ as in~\eqref{eq:reparam} and simplifying gives
\begin{equation}
\label{eq:good-rate}
\lim_{n \to \infty} \frac{1}{n} \log \Pr[\mbox{$\sigma$ is good}] 
= - \frac{\avgdeg}{2q} \big[
(1+(q-1) \lambda) \log (1+(q-1) \lambda) 
+ (q-1)(1-\lambda) \log (1-\lambda)
\big] 
\, . 
\end{equation}
Now, by the union bound, since there are at most $q^n$ balanced partitions, the probability that any good partitions exist is exponentially small whenever the function in~\eqref{eq:good-rate} is less than $-\log q$.  This tells us that the block model is distinguishable from an \ER\ graph whenever 
\[
\avgdeg > \dcondupper 
= \frac{2 q \log q}
{(1+(q-1) \lambda) \log (1+(q-1) \lambda) 
+ (q-1)(1-\lambda) \log (1-\lambda)} \, ,
\]
As noted above, the limit $\lambda = -1/(q-1)$ corresponds to the planted graph coloring problem.  In this case $\dcondupper$ is simply the first-moment upper bound on the $q$-colorability threshold, 
\[
\dcondupper = \frac{2 \log q}{-\log (1-1/q)} < 2 q \log q \, .
\]

\subsection{All good partitions are accurate} 

Next we show that, if $d > \dcondupper$, with high probability any good
partition is correlated with the planted one.  Essentially, the previous
calculation for $G(n,m)$ corresponds to counting good partitions $\tau$ which are
uncorrelated with $\sigma$, i.e., which have $\olap(\sigma, \tau)=0$.
We will show that in order for a good partition to exist, its overlap
with $\sigma$ is strictly greater than $0$.
% matrix must have some large entries, in which case its overlap with $\sigma$ is
% bounded above $1/q$.  

Given a balanced partition $\tau$, let $\ein$ and $\eout$ denote the number of
edges $(u,v)$ with $\tau_u = \tau_v$ and $\tau_u \ne \tau_v$ respectively.
As in the previous section, we say that $\tau$ is \emph{good}
if~\eqref{eq:in-close} holds, i.e., $|\ein-\mein|, |\eout-\meout| < n^{2/3}$
where $\mein$ and $\meout$ are given by~\eqref{eq:mein-def}. 
Note that the right-hand side of~\eqref{eq:good-accurate} is an increasing
function of $\beta$, and that it coincides with $\dcondupper$ when $\beta=0$.

\begin{theorem}
\label{thm:good-accurate}
Let $G$ be generated by the stochastic block model with parameters $\cin$ and
$\cout$, and let $\avgdeg$ and $\lambda$ be defined as in~\eqref{eq:avgdeg}
and~\eqref{eq:lambda}.  If $\avgdeg > \dcondupper$ then, with high probability,
any good partition has overlap at least $\beta > 0$ with the planted
partition $\sigma$, where $\beta$ is the smallest root of 
\begin{equation}
\label{eq:good-accurate}
\avgdeg = \frac{2q \bigl( h(\beta+\frac{1}{q}) + (1-\frac{1}{q}-\beta) \log (q-1) \bigr)}
{
(1+(q-1)\lambda) \log \frac{1+(q-1) \lambda}{1 + q\beta \lambda}
+ (q-1)(1-\lambda) \log \frac{(q-1)(1-\lambda)}{q - 1 - q\beta \lambda}
}
\end{equation}
where $h = -\bigl(\beta+\frac{1}{q}\bigr) \log \bigl(\beta+\frac{1}{q}\bigr) - \bigl(1-\frac{1}{q}-\beta\bigr) \log \bigl(1-\frac{1}{q}-\beta\bigr)$
is the entropy function.  Therefore, an exponential-time algorithm exists that w.h.p.\
achieves overlap at least $\beta$.
\end{theorem}

\begin{proof}
We start by conditioning on the high-probability event that $G$ has $m$ edges,
where $|m-\mm| < n^{2/3}$ and $\mm = \avgdeg n/2$.  Call the resulting model
$G_\SBM(n,m)$ (with the matrix of parameters $M$ implicit).  It consists of the
distribution over all simple graphs with $m$ edges, with probability
proportional to $\P(G \mid \sigma)$.

In analogy with the model $G'(n,m)$ defined above, we consider another version
of the block model where the $m$ edges are chosen independently as follows.  
For each edge, we first choose an ordered pair of groups $r, s$ with
probability proportional to $M_{rs}$, i.e., with probability
$\connectivity_{rs}/q$ where $\connectivity = M/(q \avgdeg)$ is the doubly
stochastic matrix defined in~\eqref{eq:connectivity}.  We then choose the
endpoints $u$ and $v$ uniformly from $\sigma^{-1}(r)$ and $\sigma^{-1}(s)$
(with replacement if $r=s$).  
Call this model $G'_\SBM(n,m)$.  In the sparse case $\avgdeg = O(1/n)$, the
resulting graph is simple with probability $\Theta(1)$, in which event it is
generated by $G_\SBM(n,m)$.  Thus any event that holds with high probability in
$G'_\SBM(n,m)$ holds with high probability in $G_\SBM(n,m)$ as well.

Now fix a balanced partition $\tau$. Let $\theta$ denote the probability that an
edge $(u,v)$ chosen in this way is within-group with respect to $\tau$. Define the $q \times q$ matrix $\overlap$ by
\[
  \alpha_{st} = \frac qn |\sigma^{-1}(s) \cap \tau^{-1}(t)|;
\]
in other words, $\overlap_{st}$ is the probability that $\tau_u=t$ if $u$ is
chosen uniformly from those with $\sigma_u=s$. 
Up to $O(1/n)$ terms, the events that $\tau_u=t$ and $\tau_v=t$ are independent.  Thus in the limit $n \to \infty$,
\begin{align*}
\theta := \Pr[\tau_u=\tau_v] 
&= \sum_{r,s,t} \Pr[\sigma_u=r \wedge \sigma_v=s \wedge \tau_u=\tau_v=t] \\
&= \frac{1}{q} \sum_{r,s,t} \connectivity_{rs} \overlap_{rt} \overlap_{st} \\
&= \frac{1}{q} \Tr \trans{\overlap} \connectivity \overlap \, ,
\end{align*}
where $\trans{}$ denotes the matrix transpose. Since $\connectivity = \lambda \eye + (1-\lambda) \frac{\J}{q}$
and $\J \overlap = \overlap \J = \J$, this gives
\[
\theta = \frac{1 + (|\overlap|^2-1) \lambda}{q} \, ,
\]
where $|\overlap|$ denotes the Frobenius norm, 
\[
  |\overlap|^2 = \Tr \trans{\overlap} \overlap = \sum_{r,s} \overlap_{rs}^2 \, .
\]
When $\tau$ and $\sigma$ are uncorrelated and $\overlap = \J/q$, we have $\theta=1/q$ as in the previous section.  When $\sigma = \tau$ and $\overlap = \eye$, we have $\theta = \cin/(q\avgdeg) = (1+(q-1)\lambda)/q$.

For $\tau$ to be good, we need $|\ein - \mein| < n^{2/3}$.  Since $|m-\mm| < n^{2/3}$ as well, up to subexponential terms the probability that $\tau$ is good is 
\[
\Pr[\Bin(\mm,\theta) = \mein] = {\mm \choose \mein} \theta^\mein (1-\theta)^{\meout} \, . 
\]
The rate at which this occurs is again a Kullback-Leibler divergence, between binomial distributions with success probabilities $\theta$ and $\mein/\mm = \cin/(q \avgdeg)$.  Following our previous calculations gives 
\begin{align}
&\lim_{n \to \infty} \frac{1}{n} \log \Pr[\Bin(\mm,\theta) = \mein] 
\label{eq:tau-good} \\
&= -\frac{\avgdeg}{2} \left(
\frac{\cin}{q \avgdeg} \log \frac{\cin}{\theta q\avgdeg} + \left( 1 - \frac{\cin}{q \avgdeg} \right) \log \frac{1-\cin/q\avgdeg}{1-\theta} 
\right) 
\nonumber \\
&= -\frac{\avgdeg}{2q} \left[
(1+(q-1)\lambda) \log \frac{1+(q-1) \lambda}{\theta q}
+ (q-1)(1-\lambda) \log \frac{(q-1)(1-\lambda)}{q(1-\theta)} 
\right] 
\nonumber \\
&= -\frac{\avgdeg}{2q} \left[
(1+(q-1)\lambda) \log \frac{1+(q-1) \lambda}{1 + (|\overlap|^2-1) \lambda}
+ (q-1)(1-\lambda) \log \frac{(q-1)(1-\lambda)}{q - 1 - (|\overlap|^2-1) \lambda}
\right] 
\, . \nonumber 
\end{align}

We pause to prove a lemma which relates the Frobenius norm to the overlap.  This bound is far from tight except in the extreme cases $\overlap = \J/q$ and $\overlap = \eye$, but it lets us derive an explicit lower bound on the overlap of a good partition. 
\begin{lemma}
\label{lem:overlap} 
$|\overlap|^2 \le 1+ q\olap(\sigma,\tau)$.
\end{lemma}

\begin{proof}
Since $\overlap$ is doubly stochastic, Birkhoff's theorem tells us it can be expressed as a convex combination of permutation matrices, 
\[
\overlap = \sum_\pi a_\pi \pi 
\quad \text{where} \quad
\sum_\pi a_\pi = 1 \, . 
\]
Thus 
\begin{align*}
	|\overlap|^2 
        = \Tr \trans{\overlap} \overlap 
	%&= \sum_{\pi,\pi'} a_\pi a_{\pi'}\Tr\pi'\pi^{-1}\nonumber \\
	= \Tr \left( \sum_\pi a_\pi \pi^{-1} \right) \overlap 
	= \sum_\pi a_\pi \Tr \pi^{-1} \overlap 
	\le \max_\pi \Tr \pi^{-1} \overlap 
	= 1+q\olap(\sigma,\tau), \,
%	&= \sum_\pi a_\pi \Tr \overlap\pi^{-1} \, . 
\end{align*}
where the last step follows from the fact that, for balanced partitions $\sigma$ and $\tau$, the overlap is a maximum, over all permutations $\pi$:
\begin{align*}
  \olap(\sigma,\tau) &= \frac{1}{n}\max_\pi \sum_{i=1}^q \left(|\sigma^{-1}(i) \cap \tau^{-1}(\pi(i))| - \frac 1n |\sigma^{-1}(i)| |\tau^{-1} (\pi(i))| \right) \\
	&= \frac{1}{q} \max_\pi \Tr \pi^{-1} \overlap - \frac{1}{q} \, , 
\end{align*}
completing the proof.
\end{proof}

The function in~\eqref{eq:tau-good} is an increasing function of $\lambda$, since as $\lambda$ increases the distributions $\Bin[\mm,q]$ and $\Bin[\mm,\cin/(q\avgdeg)]$ become closer in Kullback-Leibler distance.  Thus if $\tau$ has overlap $\beta$,  Lemma~\ref{lem:overlap} implies 
\begin{align}
\lim_{n \to \infty} &\frac{1}{n} \Pr[\mbox{$\tau$ is good}] 
\nonumber \\
&\le -\frac{\avgdeg}{2q} \left[
(1+(q-1)\lambda) \log \frac{1+(q-1) \lambda}{1 + q\beta \lambda}
+ (q-1)(1-\lambda) \log \frac{(q-1)(1-\lambda)}{q - 1 - q\beta \lambda}
\right] \, . \label{eq:tau-good-2}
\end{align}

For fixed $\sigma$, the number of balanced partitions $\tau$ with overlap matrix $\overlap$ is the number of ways to partition each group $\sigma^{-1}(r)$ so that there are $\overlap_{rs} n/q$ vertices in $\sigma^{-1}(r) \cap \tau^{-1}(s)$:
\[
\prod_{r=1}^q {n/q \choose \{ \overlap_{rs} n / q \mid 1 \le s \le q \} } 
= \prod_{r=1}^q \frac{(n/q)!}{\prod_s (\overlap_{r,s} n/q)!} \le \e^{n H(\overlap)} \, , 
\]
where $H(\overlap)$ is the average entropy of the rows of $\overlap_{rs}/q$, 
\begin{equation}
H(\overlap) = - \frac{1}{q} \sum_{r,s} \overlap_{rs} \log \overlap_{rs} \, . 
\label{eq:H-defn}
\end{equation}
By the union bound, the probability that there are any good partitions with overlap matrix $\overlap$ is exponentially small whenever the sum of $H(\overlap)$ and the right-hand side of~\eqref{eq:tau-good-2} is negative.  For a fixed overlap $\beta$, maximized by the permutation $\pi$, the entropy $H(\overlap)$ is maximized when 
\[
\overlap_{rs} = \begin{cases}
\frac{1}{q}+\beta & \mbox{if $s=\pi(r)$} \\
\frac{1}{q}-\frac{\beta}{q-1} & \mbox{if $s \ne \pi(r)$} \, , 
\end{cases}
\]
so we have 
\begin{equation}
\label{eq:entropy-bound}
H(\overlap) \le h\left(\frac{1}{q}+\beta\right) + \left(1-\frac{1}{q}-\beta\right) \log (q-1) \, . 
\end{equation}
Combining the bounds~\eqref{eq:tau-good-2} and~\eqref{eq:entropy-bound}, and requiring that their sum is at least zero, completes the proof.
\end{proof}

%For $c$ above this line, the presence of even a single good partition indicates vanishing probability that this partition is spurious---$\dcondupper$ is an upper bound on the information-theoretic transition. This gives an exponential time protocol for distinguishing \ER\ and SBM graphs in this regime: simply search through all possible partitions for one with the correct edge densities. Returning to our original objective, we need to know for which values of $q$ $\dcondupper$ dips below the Kesten-Stigum line $c = 1/\lambda^2$, as this indicates that detection is possible outside the known easy regime. \\

\subsection{Detection below the Kesten-Stigum bound}

In \S\ref{sec:intro-contribution} we commented on the asymptotic behavior of $\dcondupper$ in various regimes.  In Table~\ref{tab:1} we give, for various values of $q$, the point $\lambda^*$ at which $\dcondupper = 1/\lambda^2$; then $\dcondupper < 1/\lambda^2$ for $\lambda < \lambda^*$.  As stated above, in the limit $q \to \infty$ we have $\dcondupper = 2/\lambda$, so $\lambda^*$ tends to $1/2$.

\begin{table}
\begin{center}
$
\begin{array}{|c|c|c|c|c|c|c|c|c|c|c|c|}
\hline
q & 5 & 6 & 7 & 8 & 9 & 10 & 11 & 20 & 100 & 1000 & 10^4 \\
\lambda^* 
& -0.239
& -0.166
& -0.112
& -0.070
& -0.036
& -0.08 
& 0.014 
& 0.127
& 0.286
& 0.372
& 0.410
\\
\hline
\end{array}
$
\end{center}
\bigskip
\caption{For $\lambda < \lambda^*$ we have $\dcondupper < 1/\lambda^2$, so that community detection is information-theoretically possible below the Kesten-Stigum bound.  For $q \ge 5$, this holds in the sufficiently disassortative case, including planted graph coloring where $\lambda = -1/(q-1)$.  For $q \ge 11$, it occurs throughout the disassortative range $\lambda < 0$, and in some assortative cases.}
\label{tab:1}
\end{table}

\section{Lower bound for symmetric SBMs: Proof of lower bound in Theorem~\ref{thm:main}}
\label{sec:second}

In this section we use the general bound of Theorem~\ref{thm:non-distinguish}
to prove the part of Theorem~\ref{thm:main} involving $\dcondlower$.  In
particular, we study the quantity $Q$---defined in Definition~\ref{def:Q}---in
the case of symmetric stochastic block models.  Note that $Q$ is defined as the
maximum of a certain function over the set of doubly
stochastic matrices.  This kind of maximization problem was studied extensively
by~\cite{achlioptas-naor} on the way to proving their lower
bound on the $q$-colorability threshold, allowing us to relate this problem to
theirs.

First, note that $Q(\pi, (M - d\J)/\sqrt{2d})$ simplifies considerably in the symmetric
case, when $\pi_i = \frac 1q$ for all $i$ and $M$ is determined by only two parameters.
In this case, $\Delta_{q^2}(\pi)$ is (up to scaling) the set of doubly stochastic matrices,
while
\[
M - d\J = \lambda d \begin{pmatrix} q-1 & & -1 \\ & \ddots & \\ -1 & & q-1 \end{pmatrix}.
\]
Going back to Definition~\ref{def:Q}, we see that $Q(\pi, (M - d\J)/\sqrt{2d}) < 1$ if and only if
$\Phi(\overlap) < 0$ for all doubly stochastic $\overlap$, where
\begin{align}
  \Phi(\overlap) &= H(\overlap) - \log q + \frac{\avgdeg \lambda^2}{2}\left(|\overlap|^2 - 1\right),
  \label{eq:phi} 
\end{align}
$|\overlap|$ denoting the Frobenius norm and $H(\cdot)$ the average row entropy of $\overlap/q$
as in \eqref{eq:H-defn}. 
By Theorem~\ref{thm:non-distinguish}, if $\Phi(\overlap) < 0$ for all doubly stochastic $\overlap$
then (i) $\P_n$ and $\Q_n$ are contiguous, and (ii) $\P_n$ is non-detectable.

\subsection{Maximizing $\Phi$}

\cite{achlioptas-naor}, in the process of proving a lower bound on the
$q$-coloring threshold for \ER\ graphs, develop substantial machinery for
optimizing $\Phi$-like functions over the polytope of doubly
stochastic matrices.
Specifically, they relax the problem to maximizing over all row-stochastic
matrices, and show that the maximizer is then a mixture of uniform rows and
rows where all but one of the entries are identical.  Although their bound is
quite general, we quote here their results for the entropy.  (Note that their
definition of $H(\overlap)$ and ours differ by a factor of $q$.)

\begin{theorem}\label{thm:achlioptas-naor}
\cite[Theorem 9]{achlioptas-naor} Let $\overlap$ be doubly stochastic with $|\overlap|^2 = \rho$.  Then
\begin{equation}
H(\overlap) \le \max_{ m\in\left[0,\frac{q(q-\rho)}{q-1}\right] } 
\left\{ 
%mh(1/q) 
\frac{m}{q} \log q 
+ \left( 1-\frac{m}{q} \right) f\!\left(\frac{q\rho-m}{q(q-m)}\right)\right\} \, ,
\end{equation}
where
\begin{align*}
	f(r) = g\!\left(\frac{1 + \sqrt{(q-1)(qr - 1)}}{q}\right) + (q-1) \,g\!\left(\frac{1 - \frac{1 + \sqrt{(q-1)(qr - 1)}}{q}}{q-1}\right) 
\end{align*}
and $g(x) = -x \log x$. 
\end{theorem}

With this result in hand and using $f(1/q) = q\,g(1/q) = \log q$, we know that for all $\overlap$ with $|\overlap|^2 = \rho$, 
\begin{align*}
\Phi(\overlap) 
%\le \frac{1}{q}\left( mkh(1/q) 
%+ (q-m) f\!\left(\frac{q\rho-m}{q(q-m)}\right)\right) - \log k + \frac{\avgdeg \lambda^2}{2}(\rho - 1)
\le 
%mh(1/q) 
\max_{m \in \bigl[0,q(q-\rho)/(q-1)\bigr]} \left(1 - \frac{m}{q} \right)  \left( f\!\left(\frac{q\rho-m}{q(q-m)}\right) - f(1/q) \right)
%- \log k 
+ \frac{\avgdeg \lambda^2}{2}(\rho - 1).
\end{align*}
Achlioptas and Naor determined the value of $\avgdeg \lambda^2/2$ for which the right-hand side is less than or equal to zero for all $m \in [0,q(q-\rho)/(q-1)]$ and all $\rho \in [1,q]$.  
%Following the analysis in \cite{achlioptas-naor}, we can rearrange slightly and obtain as our necessary condition for contiguity and nonreconstructibility that
%\begin{align*}
%	\frac{\avgdeg \lambda^2}{2}(\rho - 1) \le \left(1 - \frac{m}{q} \right) 
%	\left( h(1/q) - f\!\left(\frac{q\rho - m}{q(q-m)}\right)\right)
%\end{align*}
%for $m,\rho$ constrained as above.
\begin{lemma}\label{lem:achlioptas-naor}
\cite[Proof of Theorem 7]{achlioptas-naor} 
When $\delta < (q-1) \log (q-1)$,
\begin{align*}
	\frac{\delta(\rho - 1)}{(q-1)^2} \le \left(1 - \frac{m}{q}\right)\left( f(1/q) - f\!\left(\frac{q\rho - m}{q(q-m)}\right)\right)
\end{align*}
for all $m \in [0,q(q-\rho)/(q-1)]$ and all $\rho \in [1,q]$.  
\end{lemma}
Our lower bound is an immediate corollary of this lemma. Substituting $\delta = d \lambda^2 (q-1)^2 / 2$ and solving for $d$ gives
\begin{align}
	\dcondlower = \frac{2\log(q-1)}{q-1}\frac{1}{\lambda^2} \, .
\end{align}

As we commented in~\S\ref{sec:intro-contribution}, this corresponds to the
lower bound on the $q$-colorability threshold of $G(n,\avgdeg'/n)$ where
$\avgdeg' = 2\delta = \avgdeg \lambda^2 (q-1)^2$, scaling the eigenvalue on
each edge to $\lambda$ from its value $-1/(q-1)$ for $q$-coloring.  This fits
with the Kesten-Stigum threshold as well, since the amount of information
(appropriately defined) transmitted along each edge is proportional to
$\lambda^2$ (\cite{janson2004robust}).

\section{The second moment argument: Proof of Proposition~\ref{prop:second-moment}}
\label{sec:secondmoment}
In this section, we will prove Proposition~\ref{prop:second-moment}, thereby
showing the link between the condition $Q(\pi, (M - d\J)/\sqrt{2d}) < 1$ and
the boundedness of certain second moments.
Our first lemma expresses the second moment in question in terms of (centered and
normalized) multinomial random variables.  In order to state the lemma, we make
the following notation.  Given two assignments $\sigma, \tau \in [q]^n$, let
$N_{ij} \defas N_{ij}(\sigma, \tau) \defas |\{v : \sigma_v = i, \tau_v = j\}|$,
and $X_{ij} \defas X_{ij}(\sigma, \tau) \defas n^{-1/2}\left(N_{ij}- n
\pi_i\pi_j\right)$.
Recall that $\Omega_n$ is the event that the label frequencies are approximately their expected values, and
let $Y_n$ denote the restricted density $\mathbbm{1}_{\Omega_n} \frac{d \P_n}{d \Q_n}$.
 With a slight overloading of notation, we write $\sigma \in \Omega_n$ if for all $i \in [q]$, 
 $|\{u : \sigma_u = i\}| = n \pi_i \pm a_n$.
Set $A \defas M - d \J$.

\begin{lemma}\label{lem:second-moment-simpl}
We have:
\begin{align*}
  \E_{\Q_n} Y_n^2
  &= (1 + O(n^{-1}))
 \sum_{\sigma,\tau \in \Omega_n} \P_n(\sigma) \P_n(\tau)
 \exp\left(\frac 1{2d} \sum_{ijk\ell} X_{ij} X_{k\ell} A_{ik} A_{j\ell} + \nu_1 + \nu_2 + \xi_n\right),
\end{align*}
where
\begin{align*}
  \nu_1 &= -\frac{1}{2d} \sum_{ij} A_{ii} A_{jj} \pi_i \pi_j, \\
  \nu_2 &= -\frac{1}{2d^2} \sum_{ijk\ell} A_{ik}^2 A_{j\ell}^2 \pi_i \pi_j \pi_k \pi_\ell, \mbox{ and} \\
  \xi_n &= O(n^{-1/2}) \sum_{ij} |X_{ij}| + O(n^{-1}) \left(\sum_{ij} |X_{ij}|\right)^2.
\end{align*}
\end{lemma}

\begin{proof}
 For a graph $G$ and assignment $\sigma$, define 
 \[
  W_{uv}(G, \sigma)
  = \begin{cases}
     \frac{M_{\sigma_u,\sigma_v}}{d} &\text{if $(u, v) \in E(G)$} \\
     \frac{1 - \frac{M_{\sigma_u,\sigma_v}}{n}}{1 - \frac dn} &\text{if $(u, v) \not \in E(G)$.}
    \end{cases}
 \]
 Then we may write out
 \begin{align*}
  Y_n &= \sum_{\sigma \in \Omega_n} \frac{\P_n(G, \sigma)}{\Q_n(G)} \\
  &= \sum_{\sigma \in \Omega_n} \P_n(\sigma)
  \prod_{u, v} W_{uv}(G, \sigma).
 \end{align*}
 Squaring both sides and taking expectations,
 \begin{align}
  \E_{\Q_n} Y_n^2
  &= \E_{\Q_n} \sum_{\sigma,\tau \in \Omega_n} \P_n(\sigma) \P_n(\tau) \prod_{u,v} W_{uv}(G, \sigma) W_{uv}(G, \tau) \notag \\
  &= \sum_{\sigma,\tau \in \Omega_n} \P_n(\sigma) \P_n(\tau) \prod_{u,v} \E_{\Q_n} [W_{uv}(G, \sigma) W_{uv}(G, \tau)],
  \label{eq:second-moment-1}
 \end{align}
 where the last equality holds because under $\Q_n$, and for any fixed $\sigma$,
 the variables $W_{uv}(G,\sigma)$ are independent as $u$ and $v$ vary.

Let us compute the inner expectation in~\eqref{eq:second-moment-1}.
Recall that under $\Q_n$, $(u, v) \in E(G)$ with probability $\frac dn$.
Writing (for brevity) $s$ for $M_{\sigma_u \sigma_v}$ and $t$ for $M_{\tau_u \tau_v}$, we have
\begin{align*}
 \E_{\Q_n} W_{uv}(G, \sigma) W_{uv}(G, \tau)
 &= \frac{st}{d^2} \cdot \frac dn + \frac{(1 - \frac sn)(1 - \frac tn)}{(1 - \frac dn)^2} (1 - \frac dn) \\
 &= \frac{st}{nd} + \left(1 - \frac sn\right)\left(1 - \frac tn\right)\left(1 + \frac dn + \frac{d^2}{n^2} + O(n^{-3})\right) \\
 &= 1 + \frac{(s - d)(t - d)}{nd} + \frac{(s-d)(t-d)}{n^2} + O(n^{-3})
\end{align*}
Setting $r = (s-d)(t-d)$, and using the fact that
$1 + x = \exp(x - x^2/2 + O(x^3))$, we have
\[
 \E_{\Q_n} W_{uv}(G, \sigma) W_{uv}(G, \tau)
 = \exp\left(
 \frac{r}{dn} + \frac{r}{n^2} - \frac{r^2}{2d^2 n^2} + O(n^{-3})
 \right).
\]
Now, if $(\sigma_u, \tau_u, \sigma_v, \tau_v) = (i, j, k, \ell)$
then $(s-d)(t-d) = (M_{ik} - d)(M_{j\ell} - d) = A_{ik} A_{j\ell}$. Hence,
\begin{equation}\label{eq:second-moment-2}
 \E_{\Q_n} W_{uv}(G, \sigma) W_{uv}(G, \tau)
 = \exp\left(
 \frac{A_{ik} A_{j\ell}}{dn} + \frac{A_{ik} A_{j\ell} }{n^2} - \frac{(A_{ik} A_{j\ell})^2}{2d^2 n^2} + O(n^{-3})
 \right).
\end{equation}
% where $t_{ijk\ell} \defas \frac {r_{ijk\ell}} d + \frac {r_{ijk\ell}} n - \frac{r_{ijk\ell}^2}{2d^2}$.
Let $N_{ijk\ell} = |\{\{u, v\}: \sigma_u = i, \tau_u = j, \sigma_v = k, \tau_v = \ell\}|$.
Plugging~\eqref{eq:second-moment-2} into~\eqref{eq:second-moment-1}, we have
\begin{align}
  \E_{\Q_n} Y_n^2
  &= (1 + O(n^{-1})) \sum_{\sigma,\tau \in \Omega_n} \P_n(\sigma) \P_n(\tau) \exp\left(
    \sum_{ijk\ell=1}^s N_{ijk\ell} \left(
 \frac{A_{ik} A_{j\ell}}{dn} + \frac{A_{ik} A_{j\ell} }{n^2} - \frac{(A_{ik} A_{j\ell})^2}{2d^2 n^2}
 \right)
  \right) \label{eq:second-moment-3}
\end{align}
where the $(1 + O(n^{-1}))$ term arises because $\sum_{ijk\ell} N_{ijk\ell} \le n^2$.
Applying Lemma~\ref{lem:Nijkl-Xijkl} (below) now finishes the proof.
\end{proof}

The last step in the proof of Lemma~\ref{lem:second-moment-simpl} requires us to replace
$N_{ijk\ell}$ by its normalized version, $X_{ij}$, and then rearrange the sums in~\eqref{eq:second-moment-3}.
We will do this step in slightly more generality, where we allow $N_{ijk\ell}$ to be defined
on a subset of the vertices. For the purposes of this section it suffices to consider $S = [n]$,
but the general form will be useful when we prove Theorem~\ref{thm:non-distinguish}.

\begin{lemma}\label{lem:Nijkl-Xijkl}
Let $S \subseteq [n]$ such that $\abs{S} = n - o(n)$. Further, let
\begin{align*}
N_{ijk\ell} &\defas N_{ijk\ell}(\sigma,\tau) \defas \abs{\{\{u, v\} : u,v \in S, \sigma_u = i, \tau_u = j, \sigma_v = k, \tau_v = \ell\}}, \\
N_{ij} &\defas N_{ij}(\sigma,\tau) \defas \abs{\{u : u \in S, \sigma_u = i, \tau_u = j\}} \mbox{ and,}\\
X_{ij} &\defas X_{ij}(\sigma, \tau) \defas n^{-1/2}\left(N_{ij} - n \pi_i \pi_j\right) \\
t_{ijk\ell} &\defas \frac{A_{ik}A_{j\ell}}{dn} + \frac{A_{ik} A_{j\ell}}{n^2} - \frac{(A_{ik} A_{j\ell})^2}{2 d^2 n^2}
\end{align*}
Then, we have:
\begin{align*}
    \sum_{ijk\ell} N_{ijk\ell} t_{ijk\ell} = 
\frac 1{2d} \sum_{ijk\ell} X_{ij} X_{k\ell} A_{ik} A_{j\ell} + \nu_1 + \nu_2 + \xi_n,
\end{align*}
where
\begin{align*}
  \nu_1 &= -\frac{1}{2d} \sum_{ij} A_{ii} A_{jj} \pi_i \pi_j, \\
  \nu_2 &= -\frac{1}{4d^2} \sum_{ijk\ell} A_{ik}^2 A_{j\ell}^2 \pi_i \pi_j \pi_k \pi_\ell, \mbox{ and} \\
  \xi_n &= O(n^{-1/2}) \sum_{ij} |X_{ij}| + O(n^{-1}) \left(\sum_{ij} |X_{ij}|\right)^2 + O(n^{-1}).
\end{align*}
\end{lemma}

\begin{proof}
We see that $N_{ijk\ell} = \frac{1}{2} N_{ij} N_{k\ell}$ unless $i = k$ and $j = \ell$, 
in which case $N_{ijk\ell} = \binom{N_{ij}}{2} = \frac 12 N_{ij} N_{k\ell} - \frac 12 N_{ij}$.
So, we have
\begin{equation}
    \sum_{ijk\ell} N_{ijk\ell} t_{ijk\ell}
  = \frac 12 \sum_{ijk\ell} N_{ij} N_{k\ell} t_{ijk\ell} - \frac 12 \sum_{ij} N_{ij} t_{ijij}
\label{eq:second-moment-4}
\end{equation}

Recall that $\sum_i \pi_i M_{ik} = d$ for any
fixed $k$ and $\sum_k \pi_k M_{ik} = d$ for any fixed $i$. It follows that
$\sum_i \pi_i A_{ij} = \sum_j \pi_j A_{ij} = 0$.
Hence,
\[
 \sum_i \pi_i t_{ijk\ell} = - \sum_i \pi_i \frac{(A_{ik} A_{j\ell})^2}{2 d^2 n^2}.
\]
Writing $N_{ij} = \sqrt n X_{ij} + n \pi_i \pi_j$, we have
\begin{align*}
 \sum_{ijk\ell} N_{ij} N_{k\ell} t_{ijk\ell}
 &= n \sum_{ijk\ell} X_{ij} X_{k\ell} t_{ijk\ell}
 - \sum_{ijk\ell} \frac{(A_{ik} A_{j\ell})^2}{2 d^2 n^2} \left(n^{3/2} X_{ij} \pi_k \pi_\ell + n^{3/2} X_{k\ell} \pi_i \pi_j + n^2 \pi_i \pi_j \pi_k \pi_\ell\right) \\
 &= n \sum_{ijk\ell} X_{ij} X_{k\ell} t_{ijk\ell}
 - \sum_{ijk\ell} \frac{(A_{ik} A_{j\ell})^2}{2 d^2} \pi_i \pi_j \pi_k \pi_\ell
 + O(n^{-1/2}) \sum_{ij} |X_{ij}|,
\end{align*}
Next, note that
$t_{ijk\ell} = \frac{1}{dn} A_{ik} A_{j\ell} + O(n^{-2})$, and so
\begin{multline*}
 \sum_{ijk\ell} N_{ij} N_{k\ell} t_{ijk\ell}
 = \frac{1}{d} \sum_{ijk\ell} X_{ij} X_{k\ell} A_{ik} A_{j\ell}
 - \frac{1}{2 d^2} \sum_{ijk\ell} (A_{ik} A_{j\ell})^2 \pi_i \pi_j \pi_k \pi_\ell \\
 + O(n^{-1/2}) \sum_{ij} |X_{ij}|
 + O(n^{-1}) \left(\sum_{ij} |X_{ij}|\right)^2;
\end{multline*}
we recognize the second term as $2 \nu_2$, and the last two terms as being part of $\xi_n$.
This takes care of first term in~\eqref{eq:second-moment-4}; for the
second term,
\[
 \sum_{ij} N_{ij} t_{ijij}
 = \sqrt n \sum_{ij} X_{ij} t_{ijij} + n \sum_{ij} \pi_i \pi_j t_{ijij}
 = O(n^{-1/2}) \sum_{ij} |X_{ij}| + \frac{1}{d} \sum_{ij} A_{ii} A_{jj} \pi_i \pi_j + O(n^{-1});
\]
here, the second term is $2 \nu_1$ and the others are part of $\xi_n$.
\end{proof}

The following lemma gives a simpler form for $\nu_1$ and $\nu_2$ appearing above.
In particular, this will allow us to relate $\nu_1$ and $\nu_2$ to the eigenvalues of $T$.
We define $B \defas \frac{1}{d} \diag(\pi) A = T - \pi \otimes \trans{\1}$, where $^\top$ denotes the transpose and $\otimes$ denotes the Kronecker product.
\begin{lemma}\label{lem:nu}
Let $\nu_1$ and $\nu_2$ be as in Lemma~\ref{lem:Nijkl-Xijkl}. Then, we have:
 \begin{align*}
  \nu_1 &= -\frac d2 \tr(B)^2 \\
  \nu_2 &= -\frac {d^2}{4} \tr(B^2)^2.
 \end{align*}
\end{lemma}

\begin{proof}
  Note that $A_{ii} \pi_i = d B_{ii}$. Hence,
  \[
   \nu_1 = -\frac 1{2d} \sum_{ij} A_{ii} A_{jj} \pi_i \pi_j = -\frac d2 \sum_{ij} B_{ii} B_{jj}
   = -\frac d2 \tr(B)^2.
  \]
  Similarly, since $A_{ik} \pi_i = B_{ik}$ and $A_{ik} \pi_k = A_{ki} \pi_k = B_{ki}$,
  \[
   \nu_2 = -\frac{d^2}{4} \sum_{ijk\ell} B_{ik} B_{ki} B_{j\ell} B_{\ell j}
   = -\frac{d^2}{4} \tr\big((B^{\otimes 2})^2\big)
   = -\frac{d^2}{4} \tr(B^2)^2.
  \]
\end{proof}

The following lemma shows that $\xi_n$ in Lemma~\ref{lem:Nijkl-Xijkl} is very small in an appropriate sense.
\begin{lemma}\label{lem:xi}
Let $\xi_n$ be as in Lemma~\ref{lem:Nijkl-Xijkl}. If $a_n = o(n^{1/2})$ then $\E \exp(a_n \xi_n) \to 1$.
\end{lemma}

\begin{proof}
 By the central limit theorem, each $X_{ij}$ has a limit in distribution as $n \to \infty$;
 hence $a_n \xi_n \to 0$ in probability. It is therefore enough to show that the sequence
 $\exp(a_n \xi_n)$ is uniformly integrable, but this follows from Hoeffding's inequality: since $X_{ij}$ is
 a centered, renormalized sum of independent indicator variables, Hoeffding's inequality
 implies that
 \[
   \Pr(|X_{ij}| \ge t) \le 2 e^{-t^2/2}.
 \]
 Let $X = X_{11}$; the definition of $\xi_n$ ensures that there is a constant $C$ such that
 $\xi_n$ is stochastically dominated by $C(Y + Y^2 + n^{-1})$, where $Y = n^{-1/2} q^2 |X|$.
 Hence,
 \[
   \Pr(\xi_n \ge C (t + t^2 + n^{-1})) \le \Pr(Y \ge t) \le 2 e^{-\frac{n t^2}{2 q^2}}.
 \]
 Since $q$ is a constant, this may be rearranged to state that
 \begin{equation}\label{eq:nu-bound}
   \Pr(\xi_n \ge t) \le 2 e^{-c n \min\{t, t^2\}}
 \end{equation}
 for some constant $c$ and all $t \ge 0$.
 Finally, for any $M \ge 0$
 \begin{align*}
   \E [e^{a_n \xi_n} 1_{\{e^{a_n \xi_n} \ge M\}}]
   &= \Pr(e^{a_n \xi_n} \ge M) + \int_M^\infty \Pr(e^{a_n \xi_n} \ge t)\, dt \\
   &= \Pr\left(\xi_n \ge \frac{\log M}{a_n}\right) + \int_M^\infty \Pr\left(\xi_n\ge \frac{\log t}{a_n}\right)\, dt
 \end{align*}
 If $a_n = o(n^{1/2})$ then~\eqref{eq:nu-bound} implies that both terms above converge to zero (uniformly in $n$) as
 $M \to \infty$.
\end{proof}

We now state the following three results before we prove the main result of this section.
The following proposition characterizes when the exponential of a quadratic form of a sequence of multinomial random variables is uniformly integrable.
Its proof can be found in Section~\ref{sec:UI-multinomials}.
 \begin{proposition}\label{prop:ui}
  Define $X_{ij}$ as in Lemma~\ref{lem:Nijkl-Xijkl}. Then
  \[
  \exp\left(\frac{1}{2d} \sum X_{ij} X_{k\ell} A_{ik} A_{j\ell}\right)
  \]
  is uniformly integrable if $Q(\pi, A/\sqrt{2d}) < 1$, and
  fails to be uniformly integrable if $Q(\pi, A/\sqrt{2d}) > 1$.
 \end{proposition}
 
Using H\"older's inequality, it is fairly straightforward to introduce the $\xi_n$ term:

\begin{lemma}\label{lem:UI-final}
  Define $X_{ij}$ as in Lemma~\ref{lem:Nijkl-Xijkl}. Then
  \[
  \exp\left(\frac{1}{2d} \sum X_{ij} X_{k\ell} A_{ik} A_{j\ell} + \xi_n\right)
  \]
  is uniformly integrable if $Q(\pi, A/\sqrt{2d}) < 1$, and
  fails to be uniformly integrable if $Q(\pi, A/\sqrt{2d}) > 1$.
\end{lemma}
\begin{proof}
Supposing that $Q(\pi, A/\sqrt{2d}) < 1$, we find some $\epsilon > 0$ such that
$Q(\pi, \sqrt{1+\epsilon} A / \sqrt{2d}) < 1$.
Set $a_n = n^{1/3}$ and $b_n = \frac{a_n}{a_n - 1}$ to be the H\"older conjugate of $a_n$. Setting
\begin{align}
W := \vec(X) \in \R^{q^2},
\label{eqn:defn-W}
\end{align}
H\"older's inequality and Lemma~\ref{lem:xi} give
\begin{align*}
  &\E_{\sigma,\tau} \exp\left((1+\frac{\epsilon}{2})\left(\frac 1{2d} \sum_{ijk\ell} X_{ij} X_{k\ell} A_{ik} A_{j\ell} + \xi_n\right)\right)\\
 &\le
 \left(\E_{\sigma,\tau} \exp\left(\frac{(1+\frac{\epsilon}{2})b_n}{2d} W^T (A^{\otimes 2}) W\right)\right)^{1/b_n}
 \left(\E \exp((1+\frac{\epsilon}{2})a_n \xi_n)\right)^{1/a_n} \\
 &\le
 \left(\E_{\sigma,\tau} \exp\left(\frac{(1+\frac{\epsilon}{2})b_n}{2d} W^T (A^{\otimes 2}) W\right)\right)^{1/b_n}.
\end{align*}
To check uniform integrability, we apply Proposition~\ref{prop:ui}. For sufficiently large $n$,
we have $b_n \le \frac{1 + \epsilon}{1+\frac{\epsilon}{2}}$ and
\[
 \exp\left(\frac {(1+\frac{\epsilon}{2})b_n}{2d} W^T A^{\otimes 2} W\right)
 \le \max\left\{1, \exp\left(\frac {(1 + \epsilon)}{2d} W^T A^{\otimes 2} W\right)\right\}.
\]
We see from the fact that $Q(\pi, \sqrt{1+\epsilon} A / \sqrt{2d}) < 1$ and
Proposition~\ref{prop:ui} that the right hand side above has a finite expectation.

To summarize, we have shown that if $Z = \exp(\frac{1}{2d} \sum X_{ij} X_{k\ell} A_{ik} A_{j\ell} + \xi_n)$ then
$\E Z^{(1+\epsilon/2)} < \infty$ for some $\epsilon > 0$. It follows that $Z$ is uniformly integrable, as claimed.

To show that $Q(\pi, A/\sqrt{2d}) > 1$ implies non-uniform integrability, requires an almost identical argument,
but using the reverse H\"older inequality instead of the usual H\"older inequality. We omit the details.
\end{proof}

The following lemma calculates the expected value of the exponential of a quadratic form of a Gaussian random vector.
\begin{lemma}\label{lem:gaussian-quadratic-limit}
  Take $Z \sim \normal(0, \Sigma)$, where $\Sigma = \diag(\pi)^{\otimes 2} - \bigl(\pi \otimes \pi\bigr)^{\otimes 2}$, where
  $a \otimes b$ denotes the Kronecker product of $a$ and $b$, and $a^{\otimes 2}$ denotes the outer product of $a$ with itself.
  Recall that $\lambda_i$ denote the eigenvalues of $T$, with $1=\lambda_1 \geq \abs{\lambda_2} \geq \cdots \geq \abs{\lambda_q}$.
  If $d \lambda_2^2 < 1$ then
  \[
   \E \exp\left(\frac{1}{2d} Z^T A^{\otimes 2} Z\right) = \prod_{i,j=2}^q \frac{1}{\sqrt{1-d\lambda_i\lambda_j}}.
  \]
  Otherwise, $\E \exp\left(\frac{1}{2d} Z^T A^{\otimes 2} Z\right) = \infty$.
\end{lemma}
\begin{proof}
A standard computation (see, e.g.~\cite{MathaiProvost:92}) shows that if $\mu_1, \dots, \mu_s$ denote the
eigenvalues of $\Sigma \tilde A $ then $\E \exp(Z^T \tilde A Z / 2) = \prod_i \frac{1}{\sqrt{1-\mu_i}}$.
Now,
\[
\Sigma A^{\otimes 2} = \bigl(\diag(\pi)^{\otimes 2} - \bigl(\pi \otimes \pi\bigr)^{\otimes 2}\bigr) A^{\otimes 2} 
= (\diag(\pi) A)^{\otimes 2} - (\pi \trans{\pi} A)^{\otimes 2}.
\]
Recall, however, that $A \pi = 0$. Hence, we are interested in the eigenvalues of
$(\diag(\pi) A)^{\otimes 2} = (dB)^{\otimes 2}$. Since the top eigenvalue of $T$ is 1 (with $1$
as its right-eigenvector and $\pi$ as its left-eigenvector), we see that
if $\lambda_1,\cdots,\lambda_q$ are the eigenvalues of $T$ with $\lambda_1 = 1$, then
\[
 \{d \lambda_i \lambda_j: i, j = 2, \dots, q\}
\]
are the eigenvalues of $\frac 1d \Sigma (A \otimes A)$.
\end{proof}

\begin{proof}[Proof of Proposition~\ref{prop:second-moment}]
First of all, note that
\[
 \frac{d\hat \P_n(G, \sigma)}{d\Q_n} = \frac{Y_n}{\P_n(\Omega_n)} = (1 + o(1)) Y_n.
\]
Hence, it suffices to compute the limit of $\E_{\Q_n} Y_n^2$.

From Lemma~\ref{lem:second-moment-simpl}, we see that we need to calculate the limit of the quantity
\begin{align*}
  \E_{\sigma,\tau \in \Omega_n} \exp\left(\frac 1{2d} \sum_{ijk\ell} X_{ij} X_{k\ell} A_{ik} A_{j\ell} + \xi_n\right).
\end{align*}
Lemma~\ref{lem:UI-final} establishes that the above sequence is uniformly integrable.

Now, note that $(N_{ij})_{i,j=1}^q$ is distributed as a multinomial random vector with $n$ trials and probabilities
$\pi_i \pi_j$. In particular, $\frac{1}{n}\E N_{ij} = \pi_i \pi_j$, $\frac{1}{n} \Var(N_{ij}) = \pi_i \pi_j - (\pi_i \pi_j)^2$, and
$\frac{1}{n} \Cov(N_{ij} N_{k\ell}) = -\pi_i \pi_j \pi_k \pi_\ell$ if $\{i,j\} \ne \{k,\ell\}$.
Since $X_{ij} = n^{-\frac{1}{2}}\left(N_{ij} - n \pi_i \pi_j\right)$, central limit theorem implies that $W \defas \vec(X) \in \R^{k^2}$ converges in distribution to a Gaussian random vector,
$Z$ with mean $0$ and covariance matrix $\diag(\pi)^{\otimes 2} - \pi^{\otimes 4}$.
Using Lemma~\ref{lem:gaussian-quadratic-limit} now gives us
\begin{equation}\label{eq:second-moment-with-nu}
 \E_{\Q_n} Y_n^2 \to \exp(\nu_1 + \nu_2) \prod_{i,j=2}^q \frac{1}{\sqrt{1-d\lambda_i \lambda_j}}.
\end{equation}
Going back to Lemma~\ref{lem:nu}, we have
\[
 \nu_1 = -\frac{d}{2} \tr(B)^2 = -\frac 12 \sum_{i,j=2}^q d\lambda_i \lambda_j
\]
and
\[
 \nu_2 = -\frac{d^2}{4} \tr(B^2)^2 = -\frac 14 \sum_{i,j=2}^q (d \lambda_i \lambda_j)^2.
\]
Hence, the right hand side of~\eqref{eq:second-moment-with-nu} is equal to
\[
 \prod_{i,j} \psi(d\lambda_i\lambda_j),
\]
as claimed.
\end{proof}

\section{Proof of Theorem~\ref{thm:non-distinguish}}
\label{sec:proof-mainthm}
\subsection{Non-distinguishability}
\label{sec:nondistinguish}
In this section, we use Proposition~\ref{prop:second-moment} to the contiguity claim in
Theorem~\ref{thm:non-distinguish}. Our main tool is the conditional second
moment method, which was originally developed by \cite{RW:92} in their study of Hamiltonian cycles in $d$-regular
graphs. \cite{Janson:95} was the first to apply this method for proving
contiguity.  We use a formulation from~\cite[Theorem 4.1]{Wormald}:

\begin{theorem}\label{thm:conditional-second-mom}
 Consider two sequences $\P_n, \Q_n$ of probability distributions on a sequence $\Omega_n$ of probability spaces.
 Suppose that there exist random variables $\{X_{m,n}: m \ge 3\}$, where $X_{m,n}$ is defined on $\Omega_n$,
 such that for every $m$,
 \begin{align}
  X_{m,n} &\toD \Pois(\mu_m) \text{ under $\Q_n$ as $n \to \infty$; and} \label{eq:cycles-Qn} \\
  X_{m,n} &\toD \Pois(\mu_m(1 + \delta_m)) \text{ under $\P_n$ as $n \to \infty$.}\label{eq:cycles-Pn}
 \end{align}
 Suppose also that for any $m^*$, the collection $X_{3,n}, \dots, X_{m^*,n}$ are asymptotically independent
 as $n \to \infty$ under both $\P_n$ and $\Q_n$, in the sense that every joint moment of $X_{3,n}, \dots, X_{m^*,n}$
 converges to the same joint moment of the appropriate independent Poisson variables.
 If
 \begin{equation}\label{eq:second-moment-cond}
  \E_{\Q_n} \left(\frac{\P_n}{\Q_n}\right)^2 \le (1 + o(1)) \exp\left(\sum_{m \ge 3} \mu_m \delta_m^2\right)
  < \infty
 \end{equation}
 then $\P_n$ and $\Q_n$ are contiguous.
\end{theorem}

We will apply Theorem~\ref{thm:conditional-second-mom} with $\P_n$ replaced by
$\hat \P_n = (\P_n \mid \Omega_n)$; i.e., the
block model conditioned on having almost the expected label frequencies. We will
take $X_{m,n}$ to be the number of $m$-cycles in the graph $G$ (which is drawn
either from $\hat \P_n$ or from $\Q_n$).
In order to apply Theorem~\ref{thm:conditional-second-mom}, we need to know
that the number of $m$-cycles has a limiting Poisson distribution (and we need
to know the parameters).  For $\Q_n$, this is classical; for $\P_n$ it was
proved by \cite{BoJaRi:07} (and it follows for $\hat \P_n$ since $\hat \P_n$
is obtained from $\P_n$ by conditioning on an event that holds with
probability converging to 1).

\begin{proposition}\label{prop:cycle-count}
Let $X_m$ be the number of $m$-cycles in $G$. Then
\begin{align*}
X_m &\toD \Pois\left(\frac 1{2m} d^m\right) \text{ under $\Q_n$, and}\\
X_m &\toD \Pois\left(\frac 1{2m} d^m \tr(T^m)\right) \text{ under $\P_n$.}
\end{align*}
Moreover, for any fixed $m^*$ the variables $\{X_3, \dots, X_{m^*}\}$ are asymptotically
independent under both $\P_n$ and $\Q_n$, in the sense of Theorem~\ref{thm:conditional-second-mom}.
\end{proposition}

Hence, we may apply Theorem~\ref{thm:conditional-second-mom} with $\mu_m = \frac{1}{2m} d^m$ and
$\delta_m = \tr(T^m) - 1$. Recalling that $1 = \lambda_1 \ge \cdots \ge \lambda_q$ are the eigenvalues of $T$, we have
$\delta_m = \sum_{i \ge 2} \lambda_i^m$. Hence,
\begin{align*}
 \sum_{m=3}^\infty \mu_m \delta_m^2
 &= \frac 12 \sum_{m=3}^\infty \frac {d^m}{m} \sum_{i,j=2}^q \lambda_i^m \lambda_j^m \\
 &= \frac 12 \sum_{i,j=2}^q \sum_{m=3}^\infty \frac {(d\lambda_i \lambda_j)^m}{m} \\
 &= \sum_{i,j=2}^q \log \psi(d\lambda_i \lambda_j),
\end{align*}
where $\psi(x) = (1-x)^{-1/2} e^{-x/2-x^2/4}$. In particular, condition~\eqref{eq:second-moment-cond}
follows immediately from Proposition~\ref{prop:second-moment}, which in turn proves
that $\hat \P_n$ and $\Q_n$ are contiguous. Since
$\P_n(\Omega_n) \to 1$, $\P_n$ and $\hat \P_n$ are 
contiguous also. This proves the first statement of Theorem~\ref{thm:non-distinguish}:
if $Q(\pi, (M - d \J)/\sqrt{2d}) < 1$ then $\P_n$ and $\Q_n$ are contiguous.

% \section{Non-reconstructability}

% Suppose the total number of nodes is $n$, the number of classes is $s$,
% the average degree is $d$, the edge probability matrix is $M$, the class probability
% vector is $\pi$ and finally, $T$ is a matrix such that $T_{ij} = M_{ij}\pi_j / d$.
% Let $\mathbb{P}$ denote the probability measure on graphs generated by this stochastic block model
% and $\mathbb{Q}$ denote the probability measure on graphs generated by Erdos-Renyi $G(n,d/n)$.
% 
% We will show that under the UI condition $Q\left(\pi,A/\sqrt{2d}\right) < 1$, we have:
% \begin{align*}
%   TV\left(\prob{G \middle\vert \sigma_u=a_u \mbox{ for } u\in[r]},
% \prob{G \middle\vert \sigma_u=b_u \mbox{ for } u\in[r]}\right) = o(1),
% \end{align*}
% for any two configurations $\left(a_1,a_2,\cdots, a_r\right)$ and
% $\left(b_1,b_2,\cdots, b_r\right)$ (for any fixed $r$). This shows that asymptotically,
% reconstruction better than random is not possible.
% In this section, the $\tilde{\cdot}$ quantities refer to the corresponding quantities on nodes $[n]\setminus[r]$.
% For instance $\sigtil$ refers to the restriction of $\sigma$ to $[n]\setminus[r]$, $\Ntil$ refers to the restriction of
% $N$ to $[n]\setminus[r]$ and so on.

% In this section, we prove Theorem~\ref{thm:non-reconstruct}.

\subsection{Non-detectability}
Finally, in this section we prove that if $Q(\pi, A/\sqrt{2d}) < 1$ (where $A = M - d\J$) then
$\P_n$ is non-detectable; this will complete the proof of Theorem~\ref{thm:non-distinguish}.
The following proposition is the main technical result we need.
It shows that if $Q(\pi, A/\sqrt{2d}) < 1$ then for any two fixed configurations on a finite set of nodes,
the total variation distance between the distribution on graphs conditioned on these two configurations respectively goes to zero.
\begin{proposition}\label{prop:TVdistance}
Suppose $Q\left(\pi,A/\sqrt{2d}\right) < 1$. Then, for any fixed $r > 0$, and for any two configurations $\left(a_1,a_2,\cdots, a_r\right)$ and
$\left(b_1,b_2,\cdots, b_r\right)$, we have:
\begin{align*}
  TV\left(\prob{G \middle\vert \sigma_u=a_u \mbox{ for } u\in[r]},
\prob{G \middle\vert \sigma_u=b_u \mbox{ for } u\in[r]}\right) = o(1),
\end{align*}
where $TV(\mathbb{P}_1,\mathbb{P}_2)$ denotes the total variation distance between the two distributions $\mathbb{P}_1$ and $\mathbb{P}_2$.
\end{proposition}
\begin{proof}
We will first prove the statement of the proposition with $\P_n$ replaced by
$\hat \P_n = (\P_n \mid \Omega_n)$; i.e., the
block model conditioned on having almost the expected label frequencies.

We start by using the definition of total variation distance:
\begin{align*}
& \quad TV\left(\probcond{G \middle\vert \sigma_u=a_u \mbox{ for } u\in[r]},
\probcond{G \middle\vert \sigma_u=b_u \mbox{ for } u\in[r]}\right) \\
&= \sum_G \abs{\probcond{G \middle\vert \sigma_u=a_u \mbox{ for } u\in[r]} -
\probcond{G \middle\vert \sigma_u=b_u \mbox{ for } u\in[r]}} \\
&= \sum_G \abs{\probcond{G \middle\vert \sigma_u=a_u \mbox{ for } u\in[r]} -
\probcond{G \middle\vert \sigma_u=b_u \mbox{ for } u\in[r]}} \frac{\sqrt{\probER{G}}}{\sqrt{\probER{G}}}\\
&\stackrel{(a)}{\leq} \left(\sum_G \probER{G}\right)^{1/2} \left(\sum_G \frac{\left(\probcond{G \middle\vert \sigma_u=a_u \mbox{ for } u\in[r]} -
\probcond{G \middle\vert \sigma_u=b_u \mbox{ for } u\in[r]}\right)^2}{\probER{G}}\right)^{1/2}\\
&= \left(\sum_G \frac{\left(\sum_{\sigtil} \probcond{\sigtil}\left(\probcond{G \middle\vert a,\sigtil} -
\probcond{G \middle\vert b,\sigtil}\right)\right)^2}{\probER{G}}\right)^{1/2},
\end{align*}
where $(a)$ follows from Cauchy-Schwartz inequality and $\sigtil$ denotes an assignment on $[n]\setminus [r]$.
We can expand the numerator as follows:
\begin{align*}
&\left(\sum_{\sigtil} \probcond{\sigtil}\left(\probcond{G \middle\vert a,\sigtil} -
\probcond{G \middle\vert b,\sigtil}\right)\right)^2 \\
&= \sum_{\sigtil,\tautil} \probcond{\sigtil}\probcond{\tautil} \left(
\probcond{G \middle\vert a,\sigtil}\probcond{G \middle\vert a,\tautil} + \probcond{G \middle\vert b,\sigtil}\probcond{G \middle\vert b,\tautil} \right.\\
&\qquad\qquad\qquad\qquad \left.- \probcond{G \middle\vert a,\sigtil}\probcond{G \middle\vert b,\tautil} - \probcond{G \middle\vert b,\sigtil}\probcond{G \middle\vert a,\tautil}
\right).
\end{align*}
We will now show that the value of
\begin{align*}
\sum_{\sigtil,\tautil} \probcond{\sigtil}\probcond{\tautil} \sum_G \frac{\probcond{G \middle\vert a,\sigtil}\probcond{G \middle\vert b,\tautil}}{\probER{G}},
\end{align*}
is independent of $a$ and $b$ up to an $o(1)$ error term. This will prove our claim.
 Define
 \[
  W_{uv}(G, \sigma)
  \defas \begin{cases}
     \frac{M_{\sigma_u,\sigma_v}}{d} &\text{if $(u, v) \in E(G)$,} \\
     \frac{1 - \frac{M_{\sigma_u,\sigma_v}}{n}}{1 - \frac dn} &\text{if $(u, v) \not \in E(G)$,}
    \end{cases}
 \]
and let $s_{ijk\ell} = (M_{ik} - d)(M_{j\ell} - d)/n = A_{ik} A_{j\ell} / n$,
and $t_{ijk\ell} = \frac {s_{ijk\ell}} d + \frac {s_{ijk\ell}} n - \frac{s_{ijk\ell}^2}{2d^2}$.
We have:
\begin{align}
&\sum_{\sigtil,\tautil} \probcond{\sigtil}\probcond{\tautil} \sum_G \frac{\probcond{G \middle\vert a,\sigtil}\probcond{G \middle\vert b,\tautil}}{\probER{G}}\nonumber\\
&=\sum_{\sigtil,\tautil} \probcond{\sigtil}\probcond{\tautil}\prod_{u,v \in [n]} \E_{\Q_n} [W_{uv}(G, a,\sigtil) W_{uv}(G, b,\tautil)] \nonumber\\
&= \hat\E_{\sigtil,\tautil} \prod_{u,v \in [n]\setminus [r]} \left(1+
t_{\sigtil_u\tautil_u\sigtil_v\tautil_v}+\epsilon_n\right)
% \nonumber\\
% &\qquad\qquad\qquad 
\prod_{\stackrel{u\in[r]}{v \in [n]\setminus [r]}} \left(1+
t_{a_u b_u\sigtil_v\tautil_v}+\epsilon_n\right)
% \nonumber\\
% &\quad\qquad\qquad\qquad
\prod_{u,v\in[r]} \left(1+
t_{a_u b_u a_v b_v}+\epsilon_n\right)\nonumber\\
&= \hat\E_{\sigtil,\tautil} \prod_{i,j,k,\ell \in [q]} \left(1+
t_{ijk\ell}+\epsilon_n\right)^{\Ntil_{ijk\ell}} 
% \nonumber\\
% &\quad
\prod_{\stackrel{u\in[r]}{i,j \in [q]}} \left(1+
t_{a_u b_u ij}+\epsilon_n\right)^{\Ntil_{ij}}
\prod_{u,v\in[r]} \left(1+
t_{a_u b_u a_v b_v}+\epsilon_n\right),\label{eqn:main}
\end{align}
where $\Ntil_{ijk\ell} = \left|\{\{u, v\}: \sigtil_u = i, \tautil_u = j, \sigtil_v = k, \tautil_v = \ell\}\right|$,
$\Ntil_{ij} = \left|\{v: \sigtil_v=i, \tautil_v=j\}\right|$, and $\epsilon_n = O(n^{-3})$.
% and $\Ntil_{ij} = \abs{\set{u \in [n]\setminus[r] : \sigtil_u = i, \tautil_u=j}}$ and\\
% $\Ntil_{ijk\ell} = \abs{\set{\set{u,v} \in [n]\setminus[r] \times [n]\setminus[r]: \sigtil_u = i, \tautil_u=j, \sigtil_v=k, \tautil_v=l}}$.
We first note that the last term in \eqref{eqn:main} is essentially constant:
\begin{align*}
&\prod_{u,v\in[r]} \left(1+t_{a_u b_u a_v b_v}+\epsilon_n\right)
= \prod_{u,v\in[r]} \left(1+\bigoh{\frac{1}{n}}\right)
= \left(1+\bigoh{\frac{1}{n}}\right)^{r^2} = 1+\bigoh{\frac{1}{n}}.
\end{align*}

For the second term in \eqref{eqn:main}, since $\tilde N_{ij} < n$, we have
\begin{align}
\prod_{{i,j \in [q]}} \left(1+
t_{a_u b_u ij}+\epsilon_n\right)^{\Ntil_{ij}}
&=\prod_{i,j \in [q]} \left(1+
\frac{s_{a_u b_u ij}}{d}+\bigoh{\frac{1}{n^2}}\right)^{\Ntil_{ij}} \nonumber\\
%&\stackrel{(\zeta_1)}{=}\prod_{i,j \in [q]} \exp\left(\Ntil_{ij}\left(\frac{s_{a_u b_u ij}}{d}+\bigoh{\frac{1}{n^2}}\right)\right) \nonumber\\
&= (1+o(1)) \prod_{i,j \in [q]} \exp\left(\frac{s_{a_u b_u ij}}{d} \cdot \Ntil_{ij} \right) 
%&= (1+o(1))\prod_{i,j \in [q]} \exp\left(\frac{ns_{a_u b_u ij}}{d} \cdot \left(\frac{\Ntil_{ij}}{n}-\pi_i \pi_j\right) \right) \cdot
	%\exp\left(\frac{ns_{a_u b_u ij}}{d} \cdot \pi_i \pi_j \right), \label{eqn:term2}
\end{align}
%where $(\zeta_1)$ follows from the fact that $1+x = \exp\left(x + \bigoh{x^2}\right)$ and $(\zeta_2)$ follows from the fact that $\Ntil_{ij} < n$.
On the other hand,
\begin{align*}
\prod_{i,j \in [q]} \exp\left(\frac{ns_{a_u b_u ij}}{d} \cdot \pi_i \pi_j \right)
&= \prod_{i,j \in [q]} \exp\left(\frac{\pi_i \pi_j A_{a_u i}A_{b_u j}}{d}\right) \\
%&= \exp\left(\sum_{i,j \in [q]} \frac{\pi_i \pi_j A_{a_u i}A_{b_u j}}{d}\right) \\
&= \exp\left( \frac{\left(\sum_{i\in [q]}\pi_i A_{a_u i}\right)\left(\sum_{j\in [q]}\pi_j A_{b_u j}\right)}{d}\right)
= 1,
\end{align*}
and so we may write the second term of~\eqref{eqn:main} as
\begin{align*}
\prod_{{i,j \in [q]}} \left(1+ t_{a_u b_u ij}+\epsilon_n\right)^{\Ntil_{ij}}
&= (1+o(1))\prod_{i,j \in [q]} \exp\left(\frac{ns_{a_u b_u ij}}{d} \left(\frac{\Ntil_{ij}}{n}-\pi_i \pi_j\right) \right) \\
&= (1+o(1))\prod_{i,j \in [q]} \exp\left(\frac{ns_{a_u b_u ij}}{d} \cdot \frac{\Xtil_{ij}}{\sqrt{n}}\right),
\end{align*}
where $\Xtil_{ij}\defas n^{-1/2}\left(\Ntil_{ij}-n \pi_i\pi_j\right)$.

Going back to~\eqref{eqn:main} and plugging in our estimates on the second and third terms,
%Looking now at the first two terms of \eqref{eqn:main} we obtain (using \eqref{eqn:term2}):
\begin{align*}
&\sum_{\sigtil,\tautil} \probcond{\sigtil}\probcond{\tautil} \sum_G \frac{\probcond{G \middle\vert a,\sigtil}\probcond{G \middle\vert b,\tautil}}{\probER{G}} \\
&= (1+o(1)) \hat\E_{\sigtil,\tautil} \prod_{i,j,k,\ell \in [q]} \left(1+
t_{ijk\ell}+\epsilon_n\right)^{\Ntil_{ijk\ell}}
% \\&\qquad\qquad\qquad\quad
\prod_{\stackrel{u \in [r]}{i,j \in [q]}} \exp\left(\frac{ns_{a_u b_u ij}}{d} \cdot \frac{\Xtil_{ij}}{\sqrt{n}} \right) \\
&= (1+o(1)) \hat\E_{\sigtil,\tautil} \exp\left(
  \sum_{ijk\ell} \Ntil_{ijk\ell} t_{ijk\ell}
\right)
\prod_{\stackrel{u \in [r]}{i,j \in [q]}} \exp\left(\frac{ns_{a_u b_u ij}}{d} \cdot \frac{\Xtil_{ij}}{\sqrt{n}} \right) \\
&=
(1+o(1))\hat\E_{\sigtil,\tautil}
\exp\left(\frac 1{2d} \sum_{ijk\ell} \Xtil_{ij} \Xtil_{k\ell} A_{ik} A_{j\ell} + \nu_1 + \nu_2 + \xitil_n\right)
\prod_{\stackrel{u \in [r]}{i,j \in [q]}} \exp\left(\frac{ns_{a_u b_u ij}}{d} \cdot \frac{\Xtil_{ij}}{\sqrt{n}} \right),
% \\&\qquad\qquad\qquad\quad
%\prod_{\stackrel{u \in [r]}{i,j \in [q]}} \exp\left(\frac{ns_{a_u b_u ij}}{d} \cdot \frac{\Xtil_{ij}}{\sqrt{n}} \right) \\
%&= (1+o(1))\exp\left(\nu_1+\nu_2\right)
%\hat\E_{\sigtil,\tautil}
%\exp\left(\frac 1{2d} \sum_{ijk\ell} \Xtil_{ij} \Xtil_{k\ell} A_{ik} A_{j\ell} + \xitil_n\right)
% \\&\qquad\qquad\qquad\quad
%\prod_{\stackrel{u \in [r]}{i,j \in [q]}} \exp\left(\frac{ns_{a_u b_u ij}}{d} \cdot \frac{\Xtil_{ij}}{\sqrt{n}} \right) \\
\end{align*}
where the last equality follows from Lemma~\ref{lem:Nijkl-Xijkl}.
Note that $\exp\left(\frac 1{2d} \sum_{ijk\ell} \Xtil_{ij} \Xtil_{k\ell} A_{ik} A_{j\ell} + \xitil_n\right)$ is
independent of $a$ and $b$ and from Lemma~\ref{lem:UI-final}, we also know that it is uniformly integrable.
On the other hand, since $\abs{\Xtil_{ij}} \leq \sqrt{n}$, we see that $\exp\left(\sum_{u \in [r],i,j \in [q]} \frac{ns_{a_u b_u ij}}{d}
\cdot \frac{\Xtil_{ij}}{\sqrt{n}}\right)$
is uniformly bounded; hence, the entire displayed expression above is uniformly integrable.
Since $\Xtil_{ij} \rightarrow \Normal\left(0,\pi_i\pi_j - (\pi_i\pi_j)^2\right)$, the displayed equation above
converges to a finite quantity that is independent of $a$ and $b$.
This proves the statement of the proposition with $\P_n$ replaced by
$\hat \P_n = (\P_n \mid \Omega_n)$.
Noting that
\begin{align*}
TV\left(\prob{G \middle\vert \sigma_u=a_u \mbox{ for } u\in[r]},
\probcond{G \middle\vert \sigma_u=a_u \mbox{ for } u\in[r]}\right) = o(1), \; \forall \; a
\end{align*}
gives us the desired result.
\end{proof}

As an easy consequence of Proposition~\ref{prop:TVdistance}, the posterior distribution of a single label
is essentially unchanged if we know a bounded number of other labels:
\begin{lemma}\label{lem:condTV}
Suppose $Q\left(\pi,A/\sqrt{2d}\right) < 1$. Then, for any set $S$ such that $|S|$ is a constant, $u \notin S$, we have:
\begin{align*}
  \expec{TV\left(\prob{\sigma_u \middle\vert G, \sigma_S}, \pi \right) \middle\vert \sigma_S} = o(1).
\end{align*}
\end{lemma}
\begin{proof}
\begin{align*}
  \expec{TV\left(\prob{\sigma_u \middle\vert G, \sigma_S}, \pi \right) \middle\vert \sigma_S}
  &= \sum_{\sigma_u} \prob{\sigma_u} \sum_G \abs{\frac{\prob{G \middle\vert \sigma_u, \sigma_S}}{\prob{G \middle\vert \sigma_S}} - 1}
	\prob{G \middle\vert \sigma_S} \\
  &= \sum_{i} \pi(i) TV\left(\prob{G \middle\vert \sigma_u=i, \sigma_S},\prob{G \middle\vert \sigma_S}\right) = o(1),
\end{align*}
where the last step follows from Proposition~\ref{prop:TVdistance}.
\end{proof}

Finally, we will show the non-detectability part of Theorem~\ref{thm:non-distinguish}.
By Markov's inequality, it is enough to
show that $\lim_{n \to \infty} \expec{\olap(\calA(G), \sigma)} = 0$.
We first bound $\expec{\olap(\calA(G), \sigma)}$ as follows:
\begin{align}
  \expec{\olap(\sigma,\calA(G))}
    &= \frac{1}{n}\expec{\max_\rho \sum_{i=1}^q \left(N_{i \rho(i)} (\sigma, \calA(G)) - \frac 1n N_i(\sigma) N_{\rho(i)}(\calA(G))\right)} \nonumber\\
    &\leq  \frac{1}{n}\sum_\rho \expec{\abs{\sum_{i=1}^q \left(N_{i \rho(i)} (\sigma, \calA(G)) - \frac 1n N_i(\sigma) N_{\rho(i)}(\calA(G))\right)}}.
\label{eqn:overlap-bound}
\end{align}
We will now show that each of the terms in the above summation goes to zero. Without loss of generality, let $\rho$ be the
identity map. Fix $i\in [q]$ and
consider the term $\E{\abs{\left(N_{i i} - \frac 1n N_i(\sigma) N_{i}(\calA(G))\right)}}$ (for brevity,
we suppress $\sigma, \calA(G)$ in $N_{i i} (\sigma, \calA(G))$).
Using Jensen's inequality, it is sufficient to bound
\begin{align}
\E{\left(N_{i i} - \frac 1n N_i(\sigma) N_{i}(\calA(G))\right)^2}
&= \expec{N_{ii}^2 - \frac 2n N_{ii}N_i(\sigma) N_{i}(\calA(G)) + \frac{1}{n^2} N_i^2(\sigma) N_{i}^2(\calA(G))}.
\label{eqn:Niibound}
\end{align}
We will now calculate each of the above three terms.
\begin{align}
\E N_{ii}^2 = \E \left(\sum_u \indicator{\sigma_u = i} \indicator{\calA(G)_u = i}\right)^2
&= \sum_{u,v} \E \indicator{\sigma_u = i} \indicator{\calA(G)_u = i} \indicator{\sigma_v = i} \indicator{\calA(G)_v = i} \nonumber\\
& = \sum_{u,v} \E \indicator{\sigma_u = i} \indicator{\calA(G)_u = i} \indicator{\sigma_v = i} \indicator{\calA(G)_v = i} \nonumber\\
& = \sum_{u,v} \expec{ \expec{\indicator{\sigma_u = i} \indicator{\calA(G)_u = i} \indicator{\sigma_v = i} \indicator{\calA(G)_v = i}\middle\vert G}} \nonumber\\
& = \sum_{u,v} \expec{ \expec{\indicator{\sigma_u = i} \indicator{\sigma_v = i} \middle\vert G} \indicator{\calA(G)_u = i} \indicator{\calA(G)_v = i}} \nonumber\\
&= \left(\pi(i)^2 \expec{ \indicator{\calA(G)_u = i} \indicator{\calA(G)_v = i}} + o(1) \right) n^2,
\label{eqn:Niibound1}
\end{align}
where the last step follows from Lemma~\ref{lem:condTV}.
Coming to the second term, we have:
\begin{align}
\E N_{ii} N_i(\sigma) N_{i}(\calA(G))
&= \E \left(\sum_u \indicator{\sigma_u = i} \indicator{\calA(G)_u = i}\right)
	\left(\sum_u \indicator{\sigma_u = i}\right)
	\left(\sum_u \indicator{\calA(G)_u = i}\right) \nonumber\\
&= \sum_{u,v,w} \expec{\expec{ \indicator{\sigma_u = i} \indicator{\calA(G)_u = i}
	\indicator{\sigma_v = i} \indicator{\calA(G)_w = i} \middle\vert G}} \nonumber\\
&= \sum_{u,v,w} \expec{\expec{ \indicator{\sigma_u = i} \indicator{\sigma_v = i} \middle\vert G}
	\indicator{\calA(G)_u = i} \indicator{\calA(G)_w = i}} \nonumber\\
&= \left(\pi(i)^2 \expec{ \indicator{\calA(G)_u = i} \indicator{\calA(G)_v = i}} + o(1) \right) n^3,
\label{eqn:Niibound2}
\end{align}
where the last step again follows from Lemma~\ref{lem:condTV}.
A similar argument shows that
\begin{align}
\E N_i^2(\sigma) N_{i}^2(\calA(G)) = \left(\pi(i)^2 \expec{ \indicator{\calA(G)_u = i} \indicator{\calA(G)_v = i}} + o(1) \right) n^4.
\label{eqn:Niibound3}
\end{align}
Plugging~\eqref{eqn:Niibound1},~\eqref{eqn:Niibound2} and~\eqref{eqn:Niibound3} in~\eqref{eqn:Niibound} shows that
\begin{align*}
\E{\left(N_{i i} - \frac 1n N_i(\sigma) N_{i}(\calA(G))\right)^2} = o(n^2). 
\end{align*}
This finishes the proof.

%\input{reconstruction-belowKS}
%\input{examples}
%\input{conclusion}

% \bibliography{journals,block-model,all,mark,more,zp}
\bibliographystyle{alpha}
\bibliography{block-model,all,mark,more,zp}
\newpage
\appendix
\section{UI and multinomials}\label{sec:UI-multinomials}

Here, we restate and prove Proposition~\ref{prop:ui}.
Recall that $\Delta_q$ denotes the set $\{(\alpha_1, \dots, \alpha_q) : \alpha_i \ge 0 \text{ and }\sum_i \alpha_i = 1\}$, and
that $\Delta_{q^2}(\pi)$ denotes the set of $(\alpha_{11} \dots, \alpha_{qq})$ such that
\begin{align*}
  \alpha_{ij} & \ge 0 \text{ for all $i, j$,} \\
  \sum_{i=1}^q \alpha_{ij} &= \pi_j \text{ for all $j$, and} \\
  \sum_{j=1}^q \alpha_{ij} &= \pi_i \text{ for all $i$.}
\end{align*}

In what follows, we fix an $q^2 \times q^2$ matrix $A$ and some $\pi \in \Delta_q$.
We define $p \in \Delta_{q^2}(\pi)$ by $p_{ij} = \pi_i \pi_j$ (or alternatively, $p = \pi^{\otimes 2}$),
and we take $N \sim \Multinom(n, p)$ and $X = (N - np)/\sqrt{n}$. Finally, fix a sequence $a_n$
such that $\sqrt n \ll a_n \ll n$ and
define $\Omega_n$ to be the event that
\begin{align}
 \max_j \left|\sum_i N_{ij} - n \pi_j \right| &\le a_n \label{eq:N_ij-1} \\
 \max_i \left|\sum_j N_{ij} - n \pi_i \right| &\le a_n. \label{eq:N_ij-2}
\end{align}
Note that the condition $\sqrt n \ll a_n$
ensures that the probability of $\Omega_n$ converges to 1.

\begin{proposition}\label{prop:UI-multinomial}
Define 
\[\lambda = \sup_{\alpha \in \Delta_{q^2}(\pi)} \frac{(\alpha - p)^T A (\alpha - p)}{D(\alpha, p)}.\]
If $\lambda < 1$ then
\[
 \E [\mathbbm{1}_{\Omega_n} \exp(X^T A X)] \to \E \exp (Z^T A Z) < \infty,
\]
as $n \to \infty$, where $Z \sim \Normal(0, \diag(p) - p p^T)$.
On the other hand, if $\lambda > 1$ then
\[
 \E [\mathbbm{1}_{\Omega_n}\exp(X^T A X)] \to \infty
\]
as $n \to \infty$.
\end{proposition}

\begin{lemma}\label{lem:UI-multinomial}
 For any $\epsilon > 0$, any $q = 2, 3, \dots$, and any $p \in \Delta_q$, there is a constant $C < \infty$
 such that for any $n$,
 \[
  n^{-q/2} \sum_{r_1 + \cdots + r_q = n} \exp\left(-n \epsilon\left|\frac rn - p\right|^2\right) \le C.
 \]
\end{lemma}

\begin{proof}
 We have
 \begin{align*}
  n^{-q/2} \sum_{r_1 + \cdots + r_q = n} \exp\left(-n \epsilon\left|\frac rn - p\right|^2\right)
  &\le n^{-q/2} \sum_{r_1, \dots, r_q =1}^n \exp\left(-n \epsilon\left|\frac rn - p\right|^2\right) \\
  &= \prod_{i=1}^q \left[ n^{-1/2} \sum_{r=1}^n \exp\left(-n \epsilon\left(\frac rn - p_i\right)^2\right) \right].
 \end{align*}
 The problem has now reduced to the case $q=1$; i.e., we need to show that
 \[
  n^{-1/2} \sum_{r=1}^n \exp(-n\epsilon (r/n - p)^2) < C(p, \epsilon).
 \]
 We do this by dividing the sum above into $\ell = \lceil \sqrt n\rceil$ different sums. Note that if $\frac rn \ge p$ then
 \begin{equation}\label{eq:binom-ui-1}
  \left(\frac{r + \ell}{n} - p\right)^2 = \left(\frac rn - p\right)^2 + \frac{\ell^2}{n^2} + \frac{2\ell}{n}\left(\frac rn - p\right)
  \ge \left(\frac rn - p\right)^2 + \frac 1n.
 \end{equation}
 Hence, $r \ge n p$ implies
 \[
  \exp\left(-n\epsilon\left( \frac {r + \ell}{n} - p\right)^2 \right)
  \le e^{-\epsilon} \exp\left(-n\epsilon\left( \frac {r}{n} - p\right)^2 \right).
 \]
 Stratifying the original sum into strides of length $\ell$,
 \begin{align*}
  n^{-1/2} \sum_{r = \lceil pn \rceil}^n \exp(-n\epsilon (r/n - p)^2)
  &\le n^{-1/2} \sum_{r=\lceil pn \rceil}^{\lceil pn \rceil + \ell-1} \sum_{m=0}^\infty \exp(-n\epsilon ((r + m\ell)/n - p)^2).
 \end{align*}
 Now,~\eqref{eq:binom-ui-1} implies that the inner sum may be bounded by a geometric series with initial value less than 1, and ratio $e^{-\epsilon}$. Hence,
 \[
  n^{-1/2} \sum_{r = \lceil pn \rceil}^n \exp(-n\epsilon (r/n - p)^2) \le n^{-1/2} \ell \frac{1}{1-e^{-\epsilon}},
 \]
 which is bounded. A similar argument for the case $r \le pn$ completes the proof.
\end{proof}

\begin{proof}[Proof of Proposition~\ref{prop:UI-multinomial}]
 First, recall that for any $\alpha = (\alpha_{11}, \dots, \alpha_{qq}) \in \Delta_{q^2}$, we have
 $\Pr(N = \alpha n) \asymp \exp(-nD(\alpha, p))$; this just follows from Stirling's approximation.
 Next, note that $D(\alpha, p)$ is zero only
 for $\alpha = p$, and that $D(\alpha, p)$ is strongly concave in $\alpha$. Therefore, $\lambda < 1$ implies
 that there is some $\epsilon > 0$ such that
 \[
  D(\alpha, p) \ge (1+\epsilon) (\alpha - p)^T A (\alpha - p) + \epsilon |\alpha - p|^2
 \]
 for all $\alpha \in \Delta_{q^2}(p)$. Hence, any $\alpha \in \Delta_{q^2}(p)$ satisfies
 \begin{equation}\label{eq:binom-ui-2}
  \Pr(N = \alpha n) \exp(n (1 + \epsilon)(\alpha - p)^T A (\alpha-p)) \le C \exp(-n\epsilon |\alpha - p|^2).
 \end{equation}
 Recalling the definition of $\Omega_n$, we write (with a slight abuse of notation)
 $\alpha \in \Omega_n$ if $|\max_i \sum_j \alpha_{ij} - p_i| \le n^{-1} a_n$ and similarly with $i$ and $j$ reversed.
 Note that for every $\alpha \in \Omega_n$, there is some $\tilde \alpha \in \Delta_{q^2}(\pi)$ with
 $|\alpha - \tilde \alpha|^2 = o(n^{-1})$; in particular,~\eqref{eq:binom-ui-2}
 also holds for all $\alpha \in \Omega_n$ (with a change in the constant $C$).
 Then
 \begin{align*}
  \E [\mathbbm{1}_{\Omega_n} \exp((1 + \epsilon) X^T A X)]
  &= \sum_{\alpha \in \Omega_n} \Pr(N = n \alpha) \exp\left(n (1+\epsilon) (\alpha - p)^T A (\alpha - p)\right) \\
  &\le \sum_{\alpha \in \Omega_n} \exp\left(-n \epsilon |\alpha - p|^2\right) \\
  &\le C < \infty,
 \end{align*}
 for some constant $C$ independent of $n$, where the last line follows from Lemma~\ref{lem:UI-multinomial}.
 In particular, $\exp(X^T A X)$ has $1+\epsilon$ uniformly bounded moments, and so it is uniformly integrable
 as $n \to \infty$. Since $X \toD \normal(0, \diag(p) - p p^T)$, it follows that
 $\E \exp(X^T A X) \to \E \exp(X^T A X)$.

 In the other direction, if $\lambda > 1$ then there is some $\alpha \in \Delta_{q^2}(p)$, $\alpha \ne p$ and some
 $\epsilon > 0$ such that
 $D(\alpha, p) \le (\alpha - p)^T A (\alpha - p) - 2\epsilon$. By the continuity of $D(\alpha, p)$ and
 $(\alpha - p)^T A (\alpha - p)$, we see that for sufficiently large $n$, there exists $r \in n\Delta_{q^2}(p)$ such that
 \[
  D(r/n, p) \le (r/n - p)^T A (r/n - p) - \epsilon.
 \]
 For any $n$, let $r^* = r^*(n)$ be such an $r$. Then
 \begin{align*}
  \E \exp(X^T A X)
  &\ge \Pr(N = r^*(n)) \exp\left(n (r^*/n - p)^T A (r^*/n - p)\right) \\
  &\asymp \exp\left(n \left( (r^*/n - p)^T A (r^*/n - p) - D(r^*/n, p)\right)\right) \\
  &\ge \exp\left(n \epsilon\right) \to \infty.
 \end{align*}
\end{proof}

% \begin{proposition} If $Q(\pi, A/\sqrt{2d}) < 1$ then
% \[\E_{\Q_n} Y_n^2 = (1 + o(1)) \exp\left(\sum_{q=3}^\infty \mu_q \delta_q^2\right). \]
% If $Q(\pi, A/\sqrt{2d}) > 1$ then
% \[
%  \E_{\Q_n} Y_n^2 \to \infty.
% \]
% \end{proposition}

\end{document}